\documentclass[11pt]{article}
\usepackage{amsmath}
\usepackage{amssymb}
\usepackage{amsthm}
\usepackage{enumerate}
\usepackage[colorlinks]{hyperref}
\usepackage{enumitem}
\usepackage[margin=1in]{geometry}
\usepackage{stmaryrd}
\usepackage{pgfplots, pgfplotstable, booktabs, colortbl, array,multirow}

\usepackage{tikz}

\usepackage[style=numeric-comp,giveninits=true,url=false,maxbibnames=5,backend=bibtex]{biblatex}
\bibliography{references.bib}

\theoremstyle{plain}
\newtheorem{theorem}{Theorem}[section]
\newtheorem{lemma}[theorem]{Lemma}
\newtheorem{proposition}[theorem]{Proposition}
\newtheorem{corollary}[theorem]{Corollary}

\theoremstyle{definition}
\newtheorem{definition}[theorem]{Definition}

\newtheorem{remark}[theorem]{Remark}

\DeclareMathOperator{\dv}{div}

\DeclareMathOperator{\Tr}{Tr}
\DeclareMathOperator{\Ric}{Ric}

\DeclareMathOperator{\sff}{\mathrm{I\!I}}

\newcommand{\gt}{\widetilde{g}}
\newcommand{\g}[1]{g_{#1}}

\newcommand{\textoverline}[1]{$\overline{\mbox{#1}}$}

\colorlet{darkgreen}{green!50!black}
\definecolor{darkgreen}{rgb}{0.0,0.5,0}

\begin{document}

\title{Finite element approximation of scalar curvature in arbitrary dimension}

\author{Evan S. Gawlik\thanks{Department of Mathematics, University of Hawai`i at M\textoverline{a}noa, Honolulu, HI, 96822, USA, \href{egawlik@hawaii.edu}{egawlik@hawaii.edu} } \and Michael Neunteufel\thanks{Institute for Analysis and Scientific Computing, TU Wien, Wiedner Hauptstr.~8-10, 1040 Wien, Austria, \href{michael.neunteufel@tuwien.ac.at}{michael.neunteufel@tuwien.ac.at}}}

\date{}

\maketitle

\begin{abstract}
We analyze finite element discretizations of scalar curvature in dimension $N \ge 2$.  Our analysis focuses on piecewise polynomial interpolants of a smooth Riemannian metric $g$ on a simplicial triangulation of a polyhedral domain $\Omega \subset \mathbb{R}^N$ having maximum element diameter $h$.  We show that if such an interpolant $g_h$ has polynomial degree $r \ge 0$ and possesses single-valued tangential-tangential components on codimension-1 simplices, then it admits a natural notion of (densitized) scalar curvature that converges in the $H^{-2}(\Omega)$-norm to the (densitized) scalar curvature of $g$ at a rate of $O(h^{r+1})$ as $h \to 0$, provided that either $N = 2$ or $r \ge 1$.  As a special case, our result implies the convergence in $H^{-2}(\Omega)$ of the widely used ``angle defect'' approximation of Gaussian curvature on two-dimensional triangulations, without stringent assumptions on the interpolated metric $g_h$.  We present numerical experiments that indicate that our analytical estimates are sharp.
\end{abstract}

\section{Introduction} \label{sec:intro}

Many partial differential equations that arise in mathematical physics and geometric analysis involve the Riemann curvature tensor and its contractions.  The scalar curvature $R$, which is obtained from two contractions of the Riemann curvature tensor, is particularly important;  it serves as the integrand in the Einstein-Hilbert functional from general relativity, and it appears in the governing equation for two-dimensional Ricci flow.  To approximate solutions to PDEs involving the scalar curvature, it is necessary to discretize the nonlinear differential operator that sends a Riemannian metric tensor to its scalar curvature.  The goal of this paper is to construct and analyze such discretizations in arbitrary dimension $N \ge 2$.

We are specifically interested in the setting where a smooth Riemannian metric tensor $g$ on a polyhedral domain $\Omega \subset \mathbb{R}^N$ is approximated by a piecewise polynomial \emph{Regge metric} $g_h$ on a simplicial triangulation $\mathcal{T}$ of $\Omega$ having maximum element diameter $h$.  Here, a metric is called a Regge metric on $\mathcal{T}$ if it is piecewise smooth and its tangential-tangential components are single-valued on every codimension-1 simplex in $\mathcal{T}$.  When such a metric is piecewise polynomial, it belongs to a finite element space called the \emph{Regge finite element space}~\cite{christiansen2004characterization,christiansen2011linearization,li2018regge}.   Regge metrics are not classically differentiable, so our first task will be to assign meaning to the scalar curvature of $g_h$.  Our definition, which is a natural generalization of one that is now well-established in dimension $N=2$, treats the scalar curvature of $g_h$ as a distribution and regards it as an approximation of the densitized scalar curvature of $g$, i.e.~the scalar curvature $R$ times the volume form $\omega$.  For piecewise constant Regge metrics, our definition reduces to the classical definition of the distributional densitized curvature on piecewise flat spaces~\cite{regge1961general,cheeger1984curvature}.  It is a linear combination of Dirac delta distributions supported on $(N-2)$-simplices $S$, weighted by the \emph{angle defect} at $S$: $2\pi$ minus the sum of the dihedral angles incident at $S$.  For piecewise polynomial Regge metrics of higher degree, it includes additional contributions involving the scalar curvature in the interior of each $N$-simplex and the jump in the mean curvature across each $(N-1)$-simplex.

We study the convergence of the distributional densitized scalar curvature of $g_h$ to the densitized scalar curvature of $g$ under refinement of the triangulation.  We show in Theorem~\ref{thm:conv} that in the $H^{-2}(\Omega)$-norm, this convergence takes place at a rate of $O(h^{r+1})$ when $g_h$ is an optimal-order interpolant of $g$ that is piecewise polynomial of degree $r \ge 0$, provided that either $N=2$ or $r \ge 1$.  Our numerical experiments in Section~\ref{sec:numerical} suggest that these estimates are sharp in general.

To put this convergence result into context, let us summarize some existing convergence results in the literature on finite element approximation of the scalar curvature.  We first need to assemble some notation.  

\paragraph{Notation.}  In what follows, $W^{s,p}(\Omega)$ denotes the Sobolev-Slobodeckij space of differentiability index $s \in [0,\infty)$ and integrability index $p \in [1,\infty]$, and $\|\cdot\|_{W^{s,p}(\Omega)}$ and $|\cdot|_{W^{s,p}(\Omega)}$ denote the associated norm and semi-norm, which we always take with respect to the Euclidean metric.  We denote $L^p(\Omega)=W^{0,p}(\Omega)$ and $H^s(\Omega) = W^{s,2}(\Omega)$.  For $k \in \mathbb{N}$, we denote $H^{-k}(\Omega) = (H^k_0(\Omega))'$, where $H^k_0(\Omega)$ denotes the space of functions in $H^k(\Omega)$ whose derivatives of order $0$ through $k-1$ have vanishing trace on $\partial\Omega$, and the prime denotes the dual space.  Occasionally we use weighted $L^p$ and $H^{-k}$ spaces associated with a Riemannian metric $g$, which we denote by $L^p(\Omega,g)$ and $H^{-k}(\Omega,g)$; see Section~\ref{sec:convergence} and~\cite[Equation 4.1]{gawlik2020high} for details.

If $g$ is a smooth Riemannian metric and $g_h$ is a Regge metric, then $R(g)$ denotes the scalar curvature of $g$, $(R\omega)(g)$ denotes the densitized scalar curvature of $g$, $(R\omega)_{\rm dist}(g_h)$ denotes the distributional densitized scalar curvature of $g_h$ (defined below in Definition~\ref{def:distcurv}), and $R_h^{(q)}(g_h)$ denotes the $L^2(\Omega,g_h)$-projection of $(R\omega)_{\rm dist}(g_h)$ onto the Lagrange finite element space of degree $q$.  

We also use the terms \emph{optimal-order interpolant}, \emph{canonical interpolant}, and \emph{geodesic interpolant}  below.  The first of these is a catch-all term for any piecewise polynomial interpolant $g_h$ of $g$ that belongs to the Regge finite element space and enjoys error estimates of optimal order in $W^{s,p}(T)$-norms on $N$-simplices $T$; see Definition~\ref{def:optimalorder}.  The \emph{canonical interpolant} is a specific interpolant (which is optimal-order) detailed in~\cite[Chapter 2]{li2018regge}.  The \emph{geodesic interpolant} of $g$ is the unique piecewise constant Regge metric $g_h$ with the property that the length of every edge in $\mathcal{T}$, as measured by $g_h$, agrees with the geodesic distance between the corresponding vertices in $\mathcal{T}$, as measured by $g$.    

\paragraph{Summary of existing results.}
We can now summarize some existing results about the approximation of $g$'s curvature by $g_h$'s distributional curvature.  Throughout what follows, the letter $r$ denotes the polynomial degree of $g_h$.
\begin{enumerate}
\item Cheeger, M\"uller, and Schrader~\cite[Equation (5.7) and Theorem 5.1]{cheeger1984curvature} proved that if $r=0$ and $g_h$ is the geodesic interpolant of $g$, then $(R\omega)_{\rm dist}(g_h)$ converges to $(R\omega)(g)$ in the (setwise) sense of measures at a rate of $O(h)$ in dimension $N=2$ and $O(h^{1/2})$ in dimension $N \ge 3$.  
\item Gawlik~\cite[Theorem 4.1]{gawlik2020high} proved that if $r \ge 1$, $N=2$, and $g_h$ is any optimal-order interpolant of $g$, then $R_h^{(q)}(g_h)$ converges to $R(g)$ at a rate of $O(h^r)$ in the $H^{-1}(\Omega,g)$-norm and at a rate of $O(h^{r-k-1})$ in the broken $H^k(\Omega)$-norm, $k=0,1,2,\dots,r-2$, provided that $q \ge \max\{1,r-2\}$.
\item Berchenko-Kogan and Gawlik~\cite[Corollary 6.2]{berchenko2022finite} proved that if $r \ge 1$, $N=2$, and $g_h$ is any optimal-order interpolant of $g$, then $(R\omega)_{\rm dist}(g_h)$ converges to $(R\omega)(g)$ at a rate of $O(h^r)$ in the norm $\|u\|_{V',h} = \sup_{v \in V, v \neq 0} \langle u, v \rangle_{V',V} / \|v\|_{V,h}$, where 
\begin{equation} \label{Vdef}
V = \{v \in H^1_0(\Omega) \mid \left.v\right|_T \in H^2(T) \, \forall T \in \mathcal{T}^N  \}
\end{equation} 
and $\|v\|_{V,h} = |v|_{H^1(\Omega)} + \left( \sum_{T \in \mathcal{T}^N} h_T^2 |v|_{H^2(T)}^2 \right)^{1/2}$.  Here, $h_T$ denotes the diameter of $T$, and $\mathcal{T}^N$ denotes the set of $N$-simplices in $\mathcal{T}$.
\item Gopalakrishnan, Neunteufel, Sch\"oberl, and Wardetzky~\cite[Theorem 6.5 and Corollary 6.6]{gopalakrishnan2022analysis} proved that if $r \ge 0$, $N=2$, and $g_h$ is the canonical interpolant of $g$, then $R_h^{(r+1)}(g_h)$ converges to $R(g)$ at a rate of $O(h^{r+1})$ in the $H^{-1}(\Omega,g)$-norm and at a rate of $O(h^{r-k})$ in the broken $H^k(\Omega)$-norm, $k=0,1,2,\dots,r-1$.
\end{enumerate}

\paragraph{New results.}  As one can see from above, our analysis in this paper covers two important cases that have not yet been addressed in the literature:
\begin{enumerate}
\item We prove a convergence result in the case where $N \ge 3$ and $r \ge 1$.  This opens the door to the use of piecewise polynomial Regge metrics to approximate scalar curvature in high dimensions.
\item We prove a convergence result in the case where $N=2$, $r = 0$, and $g_h$ is an arbitrary optimal-order interpolant of $g$.  This has been a longstanding gap in the literature on Gaussian curvature approximation.  Previous efforts to address the case where $N=2$ and $r=0$ have relied on subtle properties of the geodesic interpolant~\cite{cheeger1984curvature} and the canonical interpolant~\cite{gopalakrishnan2022analysis}.  Our results establish the validity of Gaussian curvature approximations involving the angle defect without stringent assumptions on the interpolated metric tensor $g_h$.
\end{enumerate}
Note that our analysis predicts no convergence at all in the $H^{-2}(\Omega)$-norm when $N \ge 3$ and $r=0$.  Our numerical experiments suggest that this result is sharp for general optimal-order interpolants.  However, for the canonical interpolant, numerical experiments suggest that $(R\omega)_{\rm dist}(g_h)$ converges to $(R\omega)(g)$ in the $H^{-2}(\Omega)$-norm at a rate of $O(h)$ when $N \ge 3$ and $r=0$.  We intend to study this superconvergence phenomenon exhibited by the canonical interpolant in future work.

\paragraph{Structure of the paper.}
Our strategy for proving convergence of $(R\omega)_{\rm dist}(g_h)$ to $(R\omega)(g)$ consists of two steps.  First, in Sections~\ref{sec:evolution}-\ref{sec:curvature}, we study the evolution of $(R\omega)_{\rm dist}(g_h)$ under deformations of the metric, leading to an integral formula for the error $(R\omega)_{\rm dist}(g_h) - (R\omega)(g)$ which reads
\begin{equation} \label{integralformula}
\langle (R\omega)_{\rm dist}(g_h) - (R\omega)(g), v \rangle_{V',V} = \int_0^1 b_h(\gt(t);\sigma,v) - a_h(\gt(t);\sigma,v) \, dt, \quad \forall v \in V.
\end{equation}
Here, $\gt(t) = (1-t)g + tg_h$, $\sigma = \frac{\partial}{\partial t} \gt(t) = g_h-g$, $V$ is the space defined in~(\ref{Vdef}), and $b_h(\gt(t);\cdot,\cdot)$ and $a_h(\gt(t);\cdot,\cdot)$ are certain metric-dependent bilinear forms.  In Section~\ref{sec:convergence}, we use techniques from finite element theory to estimate the right-hand side of~(\ref{integralformula}), leading to Theorem~\ref{thm:conv}.

The approach above is similar to the one used in dimension $N=2$ in~\cite{gawlik2020high,berchenko2022finite,gopalakrishnan2022analysis}, but there are a few important differences.  First, we work with an integral formula for the error $(R\omega)_{\rm dist}(g_h) - (R\omega)(g)$ rather than an integral formula for the curvature itself.  Previous analyses in~\cite{gawlik2020high,berchenko2022finite,gopalakrishnan2022analysis} hinged on formulas of the latter type.  Loosely speaking, in this paper we compute the evolution of the \emph{error} along a one-parameter family of Regge metrics starting at $g$ and ending at $g_h$, whereas the papers~\cite{gawlik2020high,berchenko2022finite,gopalakrishnan2022analysis} compute the evolution of the \emph{curvature} along a pair of one-parameter families of metrics: one family that starts at the Euclidean metric $\delta$ and ends at $g_h$, and one that starts at $\delta$ and ends at $g$.  The approach based on evolving the error appears to be better suited for proving optimal error estimates.

Another key aspect of our analysis is our use of the $H^{-2}(\Omega)$-norm to measure the error.  This norm is weaker than the ones used in~\cite{gawlik2020high,berchenko2022finite,gopalakrishnan2022analysis}, and it appears to be more natural for measuring the error in the curvature.  For example, for piecewise constant Regge metrics in dimension $N=2$, we show that convergence of $(R\omega)_{\rm dist}(g_h)$ to $(R\omega)(g)$ holds in the $H^{-2}(\Omega)$-norm for any optimal-order interpolant of $g$, but numerical experiments suggest that it fails to hold in stronger norms when $g_h$ is not the canonical interpolant of $g$.  A key tool that we use to prove convergence in $H^{-2}(\Omega)$ is the near-equivalence of a certain pair of metric-dependent, mesh-dependent norms on $V$; see Proposition~\ref{prop:equiv}.  This equivalence is similar to one that Walker~\cite[Theorems 4.1 and 4.3]{walker2022poincare} proved for an analogous family of mesh-dependent norms on triangulated surfaces.

\paragraph{Additional comments.}
The formula~(\ref{integralformula}) is not only useful for the error analysis, but it is also interesting in its own right.  It has a differential counterpart (see Theorem~\ref{thm:distcurvdot}) that reads
\begin{equation} \label{distcurvdotintro}
\frac{d}{dt} \langle (R\omega)_{\rm dist}(\gt(t)), v \rangle_{V',V} = b_h(\gt(t);\sigma,v) - a_h(\gt(t);\sigma,v), \quad \forall v \in V,
\end{equation}
which mimics the formula
\begin{equation} \label{curvdotintro}
\frac{d}{dt} \int_\Omega R v \omega = \int_\Omega (\dv\dv\mathbb{S}\sigma) v \omega - \int_\Omega \langle G, \sigma \rangle v \omega, \quad \forall v \in V
\end{equation}
that holds for a family of smooth Riemannian metrics $g(t)$ with densitized scalar curvature $R\omega$ and Einstein tensor $G = \Ric - \frac{1}{2}Rg$.  Here, $\mathbb{S}\sigma = \sigma-g\Tr\sigma$, and $\dv$ is the covariant divergence operator; see below for more notational details.  

The correspondence between~(\ref{distcurvdotintro}) and~(\ref{curvdotintro}) becomes even more transparent when one inspects the formulas for $b_h$ and $a_h$ (see Theorem~\ref{thm:distcurvdot}).   The bilinear form $b_h(\gt;\cdot,\cdot)$ is (up to the appearance of $\mathbb{S}$) a non-Euclidean, $N$-dimensional generalization of a bilinear form that appears in the Hellan-Herrmann-Johnson finite element method~\cite{babuska1980analysis,arnold1985mixed,brezzi1977mixed,braess2018two,braess2019equilibration,arnold2020hellan,pechstein2017tdnns,chen2018multigrid}.  It can be regarded as the integral of $\dv\dv\mathbb{S} \sigma$ against $v$, where $\dv\dv$ is interpreted in a distributional sense.  This link with the Hellan-Herrmann-Johnson method has previously been noted and used in dimension $N=2$~\cite{gawlik2020high,berchenko2022finite,gopalakrishnan2022analysis}.

The bilinear form $a_h(\gt;\cdot,\cdot)$, which is only nonzero in dimension $N \ge 3$, appears to play the role of $\int_\Omega \langle G, \sigma \rangle v \omega$, which is also only nonzero in dimension $N \ge 3$.  It gives rise to a natural way of defining the Einstein tensor in a distributional sense for Regge metrics.  We discuss this more in Section~\ref{sec:einstein}.  Among other things, we point out that the formula for $a_h$ contains a term involving the jump in the trace-reversed second fundamental form across codimension-1 simplices; the same quantity arises in studies of singular sources in general relativity, where it encodes the well-known \emph{Israel junction conditions} across a hypersurface on which stress-energy is concentrated~\cite{israel1966singular}.

There are a few other connections between our calculations and ones that appear in the physics literature.  The variation of the Gibbons-Hawking-York boundary term in general relativity~\cite{gibbons1993action,york1972role} is one example.  It has many parallels to our calculations in Section~\ref{sec:evolution_mean}, and one can undoubtedly find formulas like~(\ref{Homegadot}) in the literature after reconciling notations.  We still give a full derivation of such formulas, not only to familiarize the reader with our notation, but also to provide careful derivations that refrain from discarding total derivatives (which integrate to zero on manifolds without boundary, but not in general) and minimize the use of local coordinate calculations where possible.

\section{Evolution of geometric quantities} \label{sec:evolution}

In this section, we consider an $N$-dimensional manifold $M$ equipped with a smooth Riemannian metric $g$, and we study the evolution of various geometric quantities under deformations of $g$.

We adopt the following notation in this section.  The Levi-Civita connection associated with $g$ is denoted $\nabla$.  If $\sigma$ is a $(p,q)$-tensor field, then its covariant derivative is the $(p,q+1)$-tensor field $\nabla\sigma$, and its covariant derivative in the direction of a vector field $X$ is the $(p,q)$-tensor field $\nabla_X \sigma$.  Its trace $\Tr\sigma$ is the contraction of $\sigma$ along the first two indices, using $g$ to raise or lower indices as needed.  We denote $\dv\sigma = \Tr\nabla\sigma$ and $\Delta \sigma = \dv \nabla \sigma$.  The $g$-inner product of two $(p,q)$-tensor fields $\sigma$ and $\rho$ is denoted $\langle \sigma, \rho \rangle$. 

The volume form associated with $g$ is denoted $\omega$.  The Ricci tensor and the scalar curvature of $g$ are denoted $\Ric$ and $R$, respectively.  When we wish to emphasize their dependence on $g$, we write $\omega(g)$, $\Ric(g)$, $R(g)$, etc.

If $D$ is an embedded submanifold of $M$, then we denote by $\omega_D$ the induced volume form on $D$.  If $\sigma$ is a tensor field on $M$, then $\left.\sigma\right|_D$ denotes the pullback of $\sigma$ under the inclusion $D \hookrightarrow M$.  Later we will introduce some additional notation related to embedded submanifolds of codimension 1, like the mean curvature $H$ and second fundamental form $\sff$; see Section~\ref{sec:evolution_mean}.

We denote the exterior derivative of a differential form $\alpha$ by $d\alpha$.  If $\alpha$ is a one-form, then $\alpha^\sharp$ denotes the vector field obtained by raising indices with $g$.  If $f$ is a scalar field, then we sometimes interpret the one-form $\nabla f=df$ as the vector field $(df)^\sharp$ without explicitly writing it.

Later, in Section~\ref{sec:convergence}, we will append a subscript $g$ to many quantities like $\nabla$ and $\langle \cdot, \cdot \rangle$ to emphasize their dependence on $g$.  \emph{In that section only}, an absent subscript will generally signal that the quantity in question is computed with respect to the Euclidean metric, which we denote by $\delta$.  We say more about this notational shift in Section~\ref{sec:convergence}.

\subsection{Evolution of the densitized scalar curvature}

First we study the evolution of the densitized scalar curvature $R\omega$ under deformations of the metric.

\begin{proposition} \label{prop:curvdot}
Let $g(t)$ be a family of smooth Riemannian metrics with time derivative $\frac{\partial}{\partial t}g =: \sigma$. We have
\[
\frac{\partial}{\partial t} (R\omega) = (\dv\dv\mathbb{S}\sigma)\omega - \langle G,\sigma \rangle \omega,
\]
where $G = \Ric - \frac{1}{2}Rg$ denotes the Einstein tensor associated with $g$ and 
\[
\mathbb{S} \sigma = \sigma - g \Tr \sigma.
\]
\end{proposition}
\begin{proof}
We compute
\[
\frac{\partial}{\partial t} (R\omega) = \dot{R}\omega + R\dot{\omega}
\]
and invoke the well-known formulas~\cite[Lemma 2]{fischer1975deformations}
\[
\dot{R} = \dv \dv \sigma - \Delta \Tr \sigma - \langle \operatorname{Ric}, \sigma \rangle
\]
and~\cite[Equation 2.4]{chow2006hamilton}
\[
\dot{\omega} = \frac{1}{2} (\Tr \sigma) \omega.
\]
Since $\Delta \Tr \sigma = \dv \dv (g \Tr \sigma)$ and $\Tr \sigma = \langle g, \sigma \rangle$, the result follows.
\end{proof}

\subsection{Evolution of the mean curvature} \label{sec:evolution_mean}

Next we study the evolution of the mean curvature $H$ of a hypersurface $F$.  We assume that the tangent bundle of $F$ is trivial, so that there exists a smooth, $g$-orthonormal frame field $\tau_1,\tau_2,\dots,\tau_{N-1}$ on $F$.  (If this is not the case, then one can simply fix a point $p \in F$ and focus on a neighborhood of $p$ on which the tangent bundle is trivial.)  We let $n$ be the unit normal to $F$ so that $n,\tau_1,\tau_2,\dots,\tau_{N-1}$ forms a right-handed $g$-orthonormal frame (in the ambient manifold) at each point on $F$.  If the metric $g$ varies smoothly in time, then we assume that the vectors $n,\tau_1,\tau_2,\dots,\tau_{N-1}$ also vary smoothly in time and remain $g$-orthonormal at all times.

We use the notation
\[
\sff(X,Y) = g(\nabla_X n, Y) = -g(n, \nabla_X Y)
\]
for the second fundamental form on $F$.  Our sign convention is such that $\Tr\sff = H$, and $H$ is positive for a sphere with an outward normal vector.  We also let $\nabla_F$ and $\dv_F$ denote the surface gradient and surface divergence operators on $F$, which have the following meanings.  For a scalar field $v$,
\[
\nabla_F v = \nabla v - n\nabla_n v = \sum_{i=1}^{N-1} \tau_i \nabla_{\tau_i} v,
\]
and for a one-form $\alpha$,
\[
\dv_F \alpha = \Tr\left( \left.\nabla\alpha\right|_F \right) = \sum_{i=1}^{N-1} (\nabla_{\tau_i}\alpha)(\tau_i).
\]
Note that in the formula $\nabla_F v = \nabla v - n\nabla_n v$, we have regarded $\nabla v$ as a vector field rather than a one-form.  
Recall that the surface divergence operator satisfies the identity
\begin{equation} \label{surfacestokes}
\int_F (\dv_F \alpha)\omega_F = \int_{\partial F} \alpha(\nu_F) \omega_{\partial F} + \int_F H \alpha(n) \omega_F,
\end{equation}
where $\nu_F$ is the outward unit normal to $\partial F$ and $H$ is the mean curvature of $F$.  
\begin{proposition} \label{prop:Homegadot}
Let $g(t)$ be a family of smooth Riemannian metrics with time derivative $\frac{\partial}{\partial t}g =: \sigma$.  Let $F$ be a time-independent hypersurface with mean curvature  $H$ and induced volume form $\omega_F$.  Then
\begin{equation} \label{Homegadot}
\frac{\partial}{\partial t} (H\omega_F)  = -\frac{1}{2} \left( \left\langle \overline{\sff},  \left.\sigma\right|_F \right\rangle + (\dv\mathbb{S}\sigma)(n) + \dv_F\left(\sigma(n,\cdot) \right) - H \sigma(n,n)  \right) \omega_F,
\end{equation}
where 
\[
\overline{\sff}(X,Y) = \sff(X,Y) - Hg(X,Y)
\]
is the trace-reversed second fundamental form.
\end{proposition}
\begin{remark}
In dimension $N=2$, the formula~(\ref{Homegadot}) simplifies considerably.  Letting $\tau$ and $n$ denote the unit tangent and unit normal to $F$, we have $\nabla_\tau n = H\tau$, $-\nabla_\tau \tau = Hn$, and $\overline{\sff}(\tau,\tau) = g(\nabla_\tau n, \tau) - Hg(\tau,\tau) = H-H=0$, so $\overline{\sff}$ vanishes.  In addition,
\begin{align*}
\dv_F\left(\sigma(n,\cdot) \right) - H \sigma(n,n)
&= \nabla_\tau \left( \sigma(n,\cdot) \right) (\tau) - H \sigma(n,n) \\
&= \nabla_\tau \left( \sigma(n,\tau) \right) - \sigma(n,\nabla_\tau \tau) - H \sigma(n,n) \\
&= \nabla_\tau \left( \sigma(n,\tau) \right).
\end{align*}
Thus, in two dimensions,
\[
\frac{\partial}{\partial t} (H\omega_F)  = -\frac{1}{2}\left( (\dv\mathbb{S}\sigma)(n) + \nabla_\tau \left( \sigma(n,\tau) \right) \right)\omega_F.
\]
\end{remark}

To prove Proposition~\ref{prop:Homegadot}, we write
\begin{equation} \label{Hdot}
\dot{H} = -\sum_{i=1}^{N-1} \frac{\partial}{\partial t} g(n, \nabla_{\tau_i} \tau_i)
\end{equation}
and use the following lemmas.

\begin{lemma} \label{lemma:covdot}
For any time-dependent vector fields $X$ and $Y$,
\[
\frac{\partial}{\partial t} \nabla_Y X = \nabla_{\dot{Y}} X + \nabla_Y \dot{X} + \frac{1}{2} \left( (\nabla_X \sigma) Y + (\nabla_Y \sigma) X - (\nabla \sigma)(X, Y) \right)^\sharp,
\]
where $(\nabla \sigma)(X, Y)$ denotes the one-form $Z \mapsto (\nabla_Z \sigma)(X, Y)$, and $(\nabla_X \sigma) Y$ denotes the one-form $Z \mapsto (\nabla_X \sigma)(Y,Z)$.
\end{lemma}
\begin{proof}
In coordinates,
\[
(\nabla_Y X)^\ell = Y^j \frac{\partial X^\ell}{\partial x^j} + \Gamma^\ell_{ij} Y^j X^i,
\]
where $\Gamma^\ell_{ij}$ denote the Christoffel symbols of the second kind associated with $g$.  Thus,
\begin{align*}
\frac{\partial}{\partial t} (\nabla_Y X)^\ell 
&= \dot{Y}^j \frac{\partial X^\ell}{\partial x^j} + \Gamma^\ell_{ij} \dot{Y}^j X^i + Y^j \frac{\partial \dot{X}^\ell}{\partial x^j} + \Gamma^\ell_{ij} Y^j \dot{X}^i + \dot{\Gamma}^\ell_{ij} Y^j X^i \\
&=  (\nabla_{\dot{Y}} X)^\ell + (\nabla_Y \dot{X})^\ell + \dot{\Gamma}^\ell_{ij} Y^j X^i.
\end{align*}
Next, we recall the following formula for the rate of change of the Christoffel symbols under a metric deformation~\cite[Equation 2.23]{chow2006hamilton}:
\[
\dot{\Gamma}^\ell_{ij} 
= \frac{1}{2} g^{\ell m} \left( (\nabla_i \sigma)_{jm} + (\nabla_j \sigma)_{im} - (\nabla_m \sigma)_{ij} \right).
\]
It follows that
\begin{align*}
\dot{\Gamma}^\ell_{ij} Y^j X^i 
&= \frac{1}{2} g^{\ell m} \left( (\nabla_X \sigma)_{j m} Y^j + (\nabla_Y \sigma)_{i m} X^i - (\nabla_m \sigma)_{ij} Y^j X^i \right) \\
&=  \frac{1}{2} \left[ \left( (\nabla_X \sigma) Y + (\nabla_Y \sigma) X  -  (\nabla \sigma)(X, Y) \right) \right]^\ell.
\end{align*}
Hence,
\[
\frac{\partial}{\partial t} (\nabla_Y X)^\ell = (\nabla_{\dot{Y}} X)^\ell + (\nabla_Y \dot{X})^\ell + \frac{1}{2} \left( (\nabla_X \sigma) Y + (\nabla_Y \sigma) X -  (\nabla \sigma)(X, Y) \right)^\ell.
\]
\end{proof}

\begin{lemma} \label{lemma:ndot}
For any time-dependent vector field $X$,
\[
\frac{\partial}{\partial t} g(n,X) = \frac{1}{2} \sigma(n,n) g(n,X) + g(n,\dot{X}).
\]
\end{lemma}
\begin{proof}
Writing $X = ng(n,X) + \sum_{i=1}^{N-1} \tau_i g(\tau_i,X)$, we compute
\begin{align*}
\frac{\partial}{\partial t} g(n,X) 
&= \sigma(n,X) + g(\dot{n},X) + g(n,\dot{X}) \\
&= \sigma(n,n)g(n,X) + \sum_{i=1}^{N-1} \sigma(n,\tau_i)g(\tau_i,X) + g(\dot{n},n)g(n,X) + \sum_{i=1}^{N-1} g(\dot{n},\tau_i)g(\tau_i,X) + g(n,\dot{X}) \\
&= \left( \sigma(n,n) + g(\dot{n},n) \right) g(n,X) + \sum_{i=1}^{N-1} \left( \sigma(n,\tau_i) + g(\dot{n},\tau_i) \right) g(\tau_i,X) + g(n,\dot{X}).
\end{align*}
For each $i=1,2,\dots,N-1$, we have
\begin{align*}
0 = \frac{\partial}{\partial t} g(n,\tau_i) 
&= \sigma(n,\tau_i) + g(\dot{n},\tau_i) + g(n,\dot{\tau_i}) \\
&= \sigma(n,\tau_i) + g(\dot{n},\tau_i)
\end{align*}
since $\dot{\tau}_i$ is $g$-orthogonal to $n$.  Likewise,
\begin{align*}
0 &= \frac{\partial}{\partial t} g(n,n) 
= \sigma(n,n) + 2g(n,\dot{n}),
\end{align*}
so the result follows.
\end{proof}

We are now ready to compute the time derivative of the mean curvature $H$.  By Lemma~\ref{lemma:ndot}, we have
\begin{align}
\dot{H} 
&= -\sum_{i=1}^{N-1} \frac{\partial}{\partial t} g(n, \nabla_{\tau_i} \tau_i) \nonumber \\
&= -\sum_{i=1}^{N-1} \left[ \frac{1}{2} \sigma(n,n)g(n,\nabla_{\tau_i} \tau_i) + g\left(n, \frac{\partial}{\partial t} \nabla_{\tau_i} \tau_i \right) \right] \nonumber \\
&= \frac{1}{2} H \sigma(n,n) - \sum_{i=1}^{N-1} g\left(n, \frac{\partial}{\partial t} \nabla_{\tau_i} \tau_i \right). \label{Hdot0}
\end{align}
Using Lemma~\ref{lemma:covdot} and the symmetry of the second fundamental form, we can write the second term as
\begin{align*}
g\left(n, \frac{\partial}{\partial t} \nabla_{\tau_i} \tau_i \right)
&= g(n, \nabla_{\dot{\tau_i}} \tau_i ) + g(n, \nabla_{\tau_i} \dot{\tau_i}) + ( \nabla_{\tau_i} \sigma )(n,\tau_i) - \frac{1}{2}(\nabla_n \sigma)(\tau_i,\tau_i) \\
&= 2g(n, \nabla_{\dot{\tau_i}} \tau_i ) + ( \nabla_{\tau_i} \sigma )(n,\tau_i) - \frac{1}{2}(\nabla_n \sigma)(\tau_i,\tau_i).
\end{align*}
The first term above, when summed over $i$, can be simplified as follows.  We write $\dot{\tau}_i = \sum_{j=1}^{N-1} \tau_j g(\tau_j,\dot{\tau}_i)$ and use the linearity of $\nabla_X Y$ in $X$ to compute
\begin{align*}
2\sum_{i=1}^{N-1} g(n, \nabla_{\dot{\tau_i}} \tau_i )
&= 2\sum_{i=1}^{N-1} \sum_{j=1}^{N-1} g(n, \nabla_{\tau_j} \tau_i ) g(\tau_j,\dot{\tau}_i) \\
&= \sum_{i=1}^{N-1} \sum_{j=1}^{N-1} g(n, \nabla_{\tau_j} \tau_i ) \left( g(\tau_j,\dot{\tau}_i) + g(\dot{\tau}_j,\tau_i) \right) \\
&= -\sum_{i=1}^{N-1} \sum_{j=1}^{N-1} g(n, \nabla_{\tau_j} \tau_i ) \sigma(\tau_j,\tau_i) \\
&= \langle \sff, \left.\sigma\right|_F \rangle.
\end{align*}
Above, we used the symmetry of the second fundamental form to pass from the first line to the second, and we used the identity
\[
0 = \frac{\partial}{\partial t} g(\tau_j,\tau_i) = \sigma(\tau_j,\tau_i) + g(\tau_j,\dot{\tau}_i) + g(\dot{\tau}_j,\tau_i)
\]
to pass from the second line to the third.  Inserting these results into~(\ref{Hdot0}), we get
\begin{align}
\dot{H} 
&= \frac{1}{2}H \sigma(n,n) - \langle \sff, \left.\sigma\right|_F \rangle + \sum_{i=1}^{N-1} \left[ \frac{1}{2} (\nabla_n \sigma)(\tau_i,\tau_i) - ( \nabla_{\tau_i} \sigma )(n,\tau_i) \right]. \label{Hdot} 
\end{align}

\begin{lemma} \label{lemma:rewritten}
We have
\begin{equation} \label{rewritten}
\sum_{i=1}^{N-1} \left[ \frac{1}{2} (\nabla_n \sigma)(\tau_i,\tau_i) - ( \nabla_{\tau_i} \sigma )(n,\tau_i) \right]
= \frac{1}{2} \left( \langle \sff, \left.\sigma\right|_F \rangle - (\dv\mathbb{S}\sigma)(n) - \dv_F\left(\sigma(n,\cdot) \right) \right).
\end{equation}
\end{lemma}
\begin{proof}
The identity $0 = \nabla_{\tau_i} \left( g(n,n) \right) = 2g(n,\nabla_{\tau_i} n)$ shows that $\nabla_{\tau_i} n$ is in the span of $\{\tau_j\}_{j=1}^{N-1}$, so the first term on the right-hand side of~(\ref{rewritten}) satisfies
\begin{align} 
\langle \sff, \left.\sigma\right|_F \rangle
&= \sum_{i=1}^{N-1} \sum_{j=1}^{N-1} \sigma(\tau_j,\tau_i) g(\tau_j, \nabla_{\tau_i} n) \nonumber \\
&= \sum_{i=1}^{N-1} \sigma(\nabla_{\tau_i} n, \tau_i). \label{sffsigma}
\end{align}
The second term on the right-hand side of~(\ref{rewritten}) can be computed as follows.  Recalling that $\mathbb{S}\sigma = \sigma - g\Tr\sigma$, we have
\begin{align*}
(\dv \mathbb{S} \sigma)(n) 
&= \nabla_n (\mathbb{S}\sigma)(n, n) + \sum_{i=1}^{N-1} \nabla_{\tau_i} (\mathbb{S}\sigma)(n, \tau_i) \\
&=  (\nabla_n \sigma)(n,n) - \nabla_n (g \Tr \sigma)(n, n) + \sum_{i=1}^{N-1} \left[ (\nabla_{\tau_i} \sigma)(n,\tau_i) - \nabla_{\tau_i} (g \Tr \sigma)(n, \tau_i) \right] \\
&= (\nabla_n \sigma)(n,n) - g(n, n) \nabla_n \Tr \sigma + \sum_{i=1}^{N-1} \left[ (\nabla_{\tau_i} \sigma)(n,\tau_i) - g(n, \tau_i) \nabla_{\tau_i} \Tr \sigma \right] \\
&= (\nabla_n \sigma)(n,n) - \nabla_n \Tr \sigma + \sum_{i=1}^{N-1} (\nabla_{\tau_i} \sigma)(n,\tau_i).
\end{align*}
Since the trace commutes with covariant differentiation, 
\begin{align*}
\nabla_n \Tr \sigma 
= \Tr \nabla_n \sigma  
&= (\nabla_n \sigma)(n,n) + \sum_{i=1}^{N-1} (\nabla_n \sigma)(\tau_i,\tau_i).
\end{align*}
Thus,
\begin{equation} \label{divSsigman}
(\dv \mathbb{S} \sigma)(n) = \sum_{i=1}^{N-1} \left[ (\nabla_{\tau_i} \sigma)(n,\tau_i) - (\nabla_n \sigma)(\tau_i,\tau_i) \right].
\end{equation}
The third term on the right-hand side of~(\ref{rewritten}) is given by
\begin{align} 
\dv_F\left(\sigma(n,\cdot) \right) 
&= \sum_{i=1}^{N-1} \nabla_{\tau_i} \left( \sigma(n,\cdot) \right) (\tau_i) \nonumber \\
&=  \sum_{i=1}^{N-1} \left[ \nabla_{\tau_i} \left( \sigma(n,\tau_i) \right) - \sigma(n,\nabla_{\tau_i} \tau_i) \right]. \label{divFsigman}
\end{align}
Combining~(\ref{sffsigma}),~(\ref{divSsigman}), and~(\ref{divFsigman}), we see that
\begin{align*}
\frac{1}{2} &\left( \langle \sff, \left.\sigma\right|_F \rangle - (\dv\mathbb{S}\sigma)(n) - \dv_F\left(\sigma(n,\cdot) \right) \right) \\
&= \frac{1}{2} \sum_{i=1}^{N-1} \left[ \sigma(\nabla_{\tau_i} n, \tau_i) - (\nabla_{\tau_i} \sigma)(n,\tau_i) + (\nabla_n \sigma)(\tau_i,\tau_i) - \nabla_{\tau_i}\left(\sigma(n,\tau_i)\right) + \sigma(n,\nabla_{\tau_i} \tau_i) \right] \\
&= \frac{1}{2} \sum_{i=1}^{N-1} \left[ (\nabla_n \sigma)(\tau_i,\tau_i) - 2(\nabla_{\tau_i}\sigma)(n,\tau_i) \right].
\end{align*}
\end{proof}

Combining Lemma~\ref{lemma:rewritten} with~(\ref{Hdot}), we get
\begin{equation} \label{Hdot2}
\dot{H} =  \frac{1}{2} \left( -\langle \sff, \left.\sigma\right|_F \rangle - (\dv\mathbb{S}\sigma)(n) - \dv_F\left(\sigma(n,\cdot) \right) + H \sigma(n,n)  \right).
\end{equation}
Proposition~\ref{prop:Homegadot} now follows from the identities
\[
\frac{\partial}{\partial t}(H\omega_F) = \dot{H}\omega_F + H\dot{\omega}_F =\dot{H}\omega_F + \frac{1}{2} H \Tr\left( \left.\sigma\right|_F \right) \omega_F
\]
and
\[
\langle \sff, \left.\sigma\right|_F \rangle - H \Tr\left( \left.\sigma\right|_F \right) = \langle \overline{\sff} ,  \left.\sigma\right|_F \rangle.
\]

\subsection{Evolution of angles}

Next we study the evolution of angles under deformations of the metric.

\begin{lemma} \label{lemma:angledot}
Let $g(t)$ be a family of smooth Riemannian metrics with time derivative $\frac{\partial}{\partial t}g =: \sigma$.  Let $(\bar{n}(t),\bar{\tau}(t))$ be a pair of $g(t)$-orthonormal vectors, and let $(n(t),\tau(t))$ be another pair of $g(t)$-orthonormal vectors lying in the span of $(\bar{n}(t),\bar{\tau}(t))$.  Let $\theta(t)$ be the angle for which
\begin{align*}
\tau &= \bar{\tau}\cos\theta + \bar{n}\sin\theta, \\
n &= -\bar{\tau}\sin\theta + \bar{n}\cos\theta.
\end{align*}
Assume that these vectors vary smoothly in time, and assume that $n(t)$ (respectively, $\bar{n}(t)$) is at all times $g(t)$-orthogonal to a time-independent hypersurface $F$ (respectively, $\bar{F}$).  
Then, at all times for which $\theta \in (0,\pi)$, we have
\begin{equation} \label{angledot}
\frac{\partial}{\partial t}\theta = \frac{1}{2}\sigma(n,\tau) - \frac{1}{2} \sigma(\bar{n},\bar{\tau}).
\end{equation}
\end{lemma}
\begin{proof}
Differentiating the relation $\cos\theta = g(\bar{n},n)$ yields
\[
-\dot{\theta}\sin\theta = \frac{\partial}{\partial t} \left( g(\bar{n},n) \right).
\]
In particular, at any time $s$, we can write
\[
-\dot{\theta}(s)\sin\theta(s) = \left.\frac{\partial}{\partial t}\right|_{t=s} \left( g(t)(\bar{n}(t),n(s)) \right) +  \left.\frac{\partial}{\partial t}\right|_{t=s} \left( g(t)(\bar{n}(s),n(t)) \right) - \sigma(s)(\bar{n}(s),n(s)).
\]
Using Lemma~\ref{lemma:ndot} and suppressing the evaluations at $t=s$, we get
\begin{align*}
-\dot{\theta}\sin\theta 
&= \frac{1}{2}\sigma(\bar{n},\bar{n}) g(\bar{n},n) + \frac{1}{2}\sigma(n,n)g(n,\bar{n}) - \sigma(\bar{n},n) \\
&= \frac{1}{2}\sigma(\bar{n},\bar{n}\cos\theta-n) + \frac{1}{2}\sigma(n\cos\theta-\bar{n},n)\\
&= \frac{1}{2}\sigma(\bar{n},\bar{\tau}\sin\theta) + \frac{1}{2}\sigma(-\tau\sin\theta,n).
\end{align*}
If $\theta \in (0,\pi)$ at time $t=s$, then we can divide by $\sin\theta$ to get~(\ref{angledot}).
\end{proof}

\section{Distributional densitized scalar curvature} \label{sec:curvature}

Let $\mathcal{T}$ be a simplicial triangulation of a polyhedral domain $\Omega \subset \mathbb{R}^N$.  We use $\mathcal{T}^k$ to denote the set of all $k$-simplices in $\mathcal{T}$.  We also use $\mathring{\mathcal{T}}^k$ to denote the subset of $\mathcal{T}^k$ consisting of $k$-simplices that are not contained in the boundary of $\Omega$.  We call such simplices \emph{interior simplices}.  We call $(N-1)$-simplices \emph{faces}.

Let $g$ be a Regge metric on $\mathcal{T}$.  Recall that this means that $\left.g\right|_T$ is a smooth Riemannian metric on each $T \in \mathcal{T}^N$, and the induced metric $\left.g\right|_F$ is single-valued on each $F \in \mathring{\mathcal{T}}^{N-1}$ (and consequently the induced metric is single-valued on all lower-dimensional simplices in $\mathcal{T}$).

On each $T \in \mathcal{T}^N$, we denote by $R_T$ the scalar curvature of $\left.g\right|_T$.  On an interior face $F \in \mathring{\mathcal{T}}^{N-1}$ that lies on the boundary of two $N$-simplices $T^+$ and $T^-$, the second fundamental form on $F$, as measured by $\left.g\right|_{T^+}$, generally differs from that measured by $\left.g\right|_{T^-}$.  We denote by $\llbracket \sff \rrbracket_F$ the jump in the second fundamental form across $F$.  More precisely,
\[
\llbracket \sff \rrbracket_F(X,Y) = \left.g\right|_{T^+}(\nabla_X n^+, Y) + \left.g\right|_{T^-}(\nabla_X n^-, Y)
\]
for any vectors $X,Y$ tangent to $F$, where $n^\pm$ points outward from $T^\pm$, has unit length with respect to $\left.g\right|_{T^\pm}$, and is $\left.g\right|_{T^\pm}$-orthogonal to $F$.  We adopt similar notation for the jumps in other quantities across $F$.  For instance, $\llbracket H \rrbracket_F$ denotes the jump in the mean curvature across $F$.  We sometimes drop the subscript $F$ when there is no danger of confusion.  If $F$ is contained in $\partial\Omega$, then we define the jump in a scalar field $v$ across $F$ to be simply $\llbracket v \rrbracket_F = \left.v\right|_F$.

On each $S \in \mathring{\mathcal{T}}^{N-2}$, the \emph{angle defect} along $S$ is
\[
\Theta_S = 2\pi - \sum_{\substack{T \in \mathcal{T}^N \\ T \supset S}} \theta_{ST},
\]
where $\theta_{ST}$ denotes the dihedral angle formed by the two faces of $T$ that contain $S$, as measured by $\left.g\right|_T$. Generally this angle may vary along $S$.  If $F^+$ and $F^-$ are the two faces of $T$ that contain $S$, and if $n^\pm$ denotes the unit normal to $F^\pm$ with respect to $\left.g\right|_T$ pointing outward from $T$, then 
\[
\cos\theta_{ST} = -\left.g\right|_T(n^+,n^-).
\]

Let 
\[
V = \{v \in H^1_0(\Omega) \mid \forall T \in \mathcal{T}^N, \left.v\right|_T \in H^2(T) \}.
\]
Note that if $v \in V$, then $v$ admits a single-valued trace on every simplex in $\mathcal{T}$ of dimension $\ge N-3$.

\begin{definition} \label{def:distcurv}
Let $g$ be a Regge metric.  The \emph{distributional densitized scalar curvature} of $g$ is the linear functional $(R\omega)_{\rm dist}(g) \in V'$ defined by
\begin{equation} \label{distcurv}
\langle (R\omega)_{\rm dist}(g), v \rangle_{V',V} = \sum_{T \in \mathcal{T}^N} \int_T R_T v \omega_T + 2\sum_{F \in \mathring{\mathcal{T}}^{N-1}} \int_F \llbracket H \rrbracket_F v \omega_F + 2\sum_{S \in \mathring{\mathcal{T}}^{N-2}} \int_S \Theta_S v \omega_S, \quad \forall v \in V.
\end{equation}
\end{definition}

This definition generalizes Definition 3.1 of~\cite{berchenko2022finite}, where the distributional curvature two-form (i.e. the Gaussian curvature times the volume form) is defined for Regge metrics in dimension $N=2$.  Note that the factors of 2 appearing in all but the first term in~(\ref{distcurv}) are consistent with the fact that in dimension $N=2$, the scalar curvature $R$ is twice the Gaussian curvature.  

One can heuristically motivate Definition~\ref{def:distcurv} in much the same way that one motivates its two-dimensional counterpart. When $g$ is piecewise constant, Definition~\ref{def:distcurv} recovers the classical notion~\cite{regge1961general} that the distributional densitized scalar curvature is a linear combination of Dirac delta distributions supported on $(N-2)$-simplices, with weights given by angle defects.  When $g$ is not piecewise constant, additional terms appear which account for the nonzero (classically defined) curvature of $g$ in the interior of each $N$-simplex $T$ and the jump in the mean curvature across each interior face $F$.  The jump in the mean curvature across $F$ can be understood by recalling that the scalar curvature $R$ at a point $p \in F$ can be expressed as (two times) a sum of sectional curvatures of $N(N-1)/2$ tangent planes that are mutually $g$-orthogonal at $p$, $(N-1)(N-2)/2$ of which are tangent to $F$ at $p$ and $N-1$ of which are $g$-orthogonal to $F$ at $p$.  The sectional curvatures corresponding to planes tangent to $F$ are nonsingular, owing to the tangential-tangential continuity of $g$.  The remaining $N-1$ sectional curvatures are singular, and by considering an $N$-dimensional region that encloses a portion of $F$ and has small thickness in the direction that is $g$-orthogonal of $F$, one can use the Gauss-Bonnet theorem (along two-dimensional slices) to approximate the (volume-)integrated sum of these sectional curvatures by the (surface-)integrated jump in the mean curvature across $F$. (In this calculation, one must bear in mind that sectional curvatures and Gaussian curvatures are related via the Gauss-Codazzi equations.) See the discussion after Definition 3.1 in~\cite{berchenko2022finite}, as well as~\cite{strichartz2020defining}, for more insight in dimension $N=2$.  See also~\cite{christiansen2013exact} for a justification of Definition 3.1 in the case where $g$ is piecewise constant and $N \ge 2$.

In the sequel, we will consistently use the letters $T$, $F$, and $S$ to refer to simplices of dimension $N$, $N-1$, and $N-2$, respectively.  We will therefore write $\sum_T$, $\sum_F$, and $\sum_S$ in place of $\sum_{T \in \mathcal{T}^N}$, $\sum_{F \in \mathcal{T}^{N-1}}$, and $\sum_{S \in \mathcal{T}^{N-2}}$, respectively.  When we wish to sum over \emph{interior} simplices of a given dimension, we put a ring on top of the summation symbol.  Thus, for example, $\mathring{\sum}_F$ is shorthand for $\sum_{F \in \mathring{\mathcal{T}}^{N-1}}$.

\subsection{Evolution of the distributional scalar curvature}

We are interested in how~(\ref{distcurv}) changes under deformations of the metric.  To this end, consider a one-parameter family of Regge metrics $g(t)$ with time derivative
\[
\sigma = \frac{\partial}{\partial t}g.
\] 
Our goal will be to compute
\[
\frac{d}{dt} \langle (R\omega)_{\rm dist}(g(t)), v \rangle_{V',V}
\]
with $v \in V$ arbitrary.

According to Propositions~\ref{prop:curvdot} and~\ref{prop:Homegadot}, the derivatives of the first two terms on the right-hand side of~(\ref{distcurv}) satisfy
\[
\frac{d}{dt}\int_T R_T v \omega_T = \int_T \left( \dv\dv\mathbb{S}\sigma - \langle G,\sigma \rangle \right) v \omega_T
\]
and
\begin{align}
2 \frac{d}{dt} \int_F \llbracket H \rrbracket_F v \omega_F 
&= -\int_F \left\llbracket \left\langle \overline{\sff},  \left.\sigma\right|_F \right\rangle + (\dv\mathbb{S}\sigma)(n) + \dv_F\left(\sigma(n,\cdot) \right) - H \sigma(n,n)  \right\rrbracket v \omega_F. \label{Hdot2terms}
\end{align}
For the third term on the right-hand side of~(\ref{distcurv}), we use the following lemma.

\begin{lemma} \label{lemma:angledefectdot}
Along any interior $(N-2)$-simplex $S$, we have
\[
\frac{\partial}{\partial t} (\Theta_S \omega_S) = \frac{1}{2}
\left(\sum_{F \supset S} \llbracket  \sigma(n,\tau) \rrbracket_F  + \Theta_S \Tr(\sigma|_S) \right) \omega_S,
\]
where the sum is over all $(N-1)$-simplices $F$ that contain $S$, $n$ is the unit normal to $F$ with respect to $g$, and $\tau$ is the unit vector with respect to $g$ that points into $F$ from $S$ and is $g$-orthogonal to both $S$ and $n$.  Here, our convention is that if $F$ is shared by two $N$-simplices $T^+$ and $T^-$, then
\[
\llbracket  \sigma(n,\tau) \rrbracket_F = \sigma^+(n^+,\tau) + \sigma^-(n^-,\tau),
\]
where $\sigma^\pm = \left.\sigma\right|_{T^\pm}$ and $n^\pm$ points outward from $T^\pm$.
\end{lemma}
\begin{remark}
Note that $n$ generally differs on either side of $F$, whereas $\tau$ does not, because $g$ has single-valued tangential-tangential components along $F$.
\end{remark}
\begin{proof}
We compute
\[
\dot{\Theta}_S = -\sum_{T \supset S} \dot{\theta}_{ST}
\]
and use Lemma~\ref{lemma:angledot} to differentiate each angle $\theta_{ST}$.  The resulting expression for $\dot{\Theta}_S$ involves differences between $\sigma(n,\tau)$ evaluated on consecutive pairs of faces $F$ emanating from $S$.  This sum can be rearranged to give
\begin{equation} \label{angledefectdot}
\dot{\Theta}_S = \frac{1}{2} \sum_{F \supset S} \llbracket  \sigma(n,\tau) \rrbracket_F.
\end{equation}
We thus get
\begin{align*}
\frac{\partial}{\partial t} (\Theta_S \omega_S) 
&= \dot{\Theta}_S \omega_S  + \Theta_S \dot{\omega}_S \\
&= \frac{1}{2} \sum_{F \supset S} \llbracket  \sigma(n,\tau) \rrbracket_F \omega_S + \frac{1}{2} \Theta_S \Tr\left( \left.\sigma\right|_S \right) \omega_S.
\end{align*}
\end{proof}

It follows from the above lemma that
\begin{align*}
2 \frac{d}{dt} \int_S \Theta_S v \omega_S 
&= \int_S \sum_{F \supset S} \llbracket  \sigma(n,\tau) \rrbracket_F v \omega_S + \int_S \Theta_S \Tr(\sigma|_S) v \omega_S \\
&= \int_S \sum_{F \supset S} \llbracket  \sigma(n,\tau) \rrbracket_F v \omega_S + \int_S \left\langle \Theta_S g|_S, \sigma|_S \right\rangle v \omega_S.
\end{align*}

Collecting our results, we obtain
\begin{align}
\frac{d}{dt}&\langle (R\omega)_{\rm dist}(g(t)), v \rangle_{V',V} 
= \sum_T \int_T (\dv\dv\mathbb{S}\sigma)v \omega_T \nonumber
\\& - \mathring{\sum_F} \int_F \left\llbracket (\dv\mathbb{S}\sigma)(n) + \dv_F\left(\sigma(n,\cdot) \right) - H \sigma(n,n) \right\rrbracket_F v \omega_F + \mathring{\sum_S} \int_S \sum_{F \supset S} \llbracket  \sigma(n,\tau) \rrbracket_F v\omega_S \label{distcurvdot0} \\
&- \sum_T \int_T \langle G, \sigma \rangle v\omega_T - \mathring{\sum_F} \int_F \left\langle \llbracket \overline{\sff} \rrbracket_F,  \left.\sigma\right|_F \right\rangle v\omega_F + \mathring{\sum_S} \int_S \left\langle \Theta_S g|_S, \sigma|_S \right\rangle v \omega_S. \nonumber
\end{align}

We will now use integration by parts to rewrite the first three terms in a way that involves no derivatives of $\sigma$.

\begin{lemma} \label{lemma:IBPdivdivS}
For any $v \in V$, we have
\begin{align*}
\sum_T \int_T & (\dv\dv\mathbb{S}\sigma)v \omega_T - \mathring{\sum_F} \int_F \left\llbracket (\dv\mathbb{S}\sigma)(n) + \dv_F\left(\sigma(n,\cdot) \right) - H \sigma(n,n) \right\rrbracket_F v \omega_F 
\\& + \mathring{\sum_S} \int_S \sum_{F \supset S} \llbracket  \sigma(n,\tau) \rrbracket_F v\omega_S =  \sum_T \int_T \langle \mathbb{S}\sigma, \nabla\nabla v \rangle \omega - \sum_F \int_F \mathbb{S}\sigma(n,n) \llbracket \nabla_n v \rrbracket \omega_F.
\end{align*}
\end{lemma}
\begin{proof}
We have
\begin{align}
\sum_T &\int_T \langle \mathbb{S}\sigma, \nabla\nabla v \rangle \omega - \sum_F \int_F \mathbb{S}\sigma(n,n) \llbracket \nabla_n v \rrbracket \omega_F \\
&= \sum_T \bigg( \int_T \langle \mathbb{S} \sigma, \nabla \nabla v \rangle \omega -  \int_{\partial T} \mathbb{S}\sigma(n,n) \nabla_n v \, \omega_{\partial T} \bigg) \nonumber \\
&= \sum_T \bigg( \int_{\partial T} \mathbb{S}\sigma(n,\nabla v) \omega_{\partial T} - \int_T (\dv\mathbb{S}\sigma)(\nabla v) \omega -  \int_{\partial T} \mathbb{S}\sigma(n,n) \nabla_n v \, \omega_{\partial T} \bigg) \nonumber \\
&= \sum_T \bigg( \int_{\partial T} \mathbb{S}\sigma(n,\nabla v) \omega_{\partial T} - \int_{\partial T} (\dv\mathbb{S}\sigma)(n) v \omega_{\partial T} + \int_T (\dv\dv\mathbb{S}\sigma)v \omega \nonumber\\&\quad\quad - \int_{\partial T} \mathbb{S}\sigma(n,n) \nabla_n v \, \omega_{\partial T} \bigg). \label{IBP1}
\end{align}
Note that here we are regarding $\nabla v$ as a vector field rather than a one-form.
On each $N$-simplex $T$, we can write $\int_{\partial T} \mathbb{S}\sigma(n,\nabla v) \omega_{\partial T} - \int_{\partial T} \mathbb{S}\sigma(n,n) \nabla_n v \, \omega_{\partial T}$ as a sum of integrals over faces $F \subset \partial T$:
\begin{align*}
\int_{\partial T} \mathbb{S}\sigma(n,\nabla v) \omega_{\partial T} - \int_{\partial T} \mathbb{S}\sigma(n,n) \nabla_n v \, \omega_{\partial T} 
&= \sum_{F \subset \partial T} \int_F \mathbb{S}\sigma(n,\nabla v - n \nabla_n v) \omega_F \\
&= \sum_{F \subset \partial T} \int_F \mathbb{S}\sigma(n,\nabla_F v) \omega_F \\
&= \sum_{F \subset \partial T} \int_F \sigma(n,\nabla_F v) \omega_F.
\end{align*}
In the last line above, we used the fact that $\nabla_F v$ is $g$-orthogonal to $n$, so
\[
\mathbb{S}\sigma(n,\nabla_F v) = \sigma(n,\nabla_F v) - g(n,\nabla_F v) \Tr \sigma = \sigma(n,\nabla_F v).
\]
Each integral over $F$ can be integrated by parts as follows.  We have 
\[
\sigma(n,\nabla_F v) = \dv_F\left( \sigma(n,\cdot)v \right) - \dv_F\left(\sigma(n,\cdot)\right) v,
\]
so the identity~(\ref{surfacestokes}) applied to $\alpha = \sigma(n,\cdot)v$ implies that
\[
\int_F \sigma(n,\nabla_F v) \omega_F = \int_{\partial F} \sigma(n,\nu_F) v \omega_{\partial F} - \int_F \left( \dv_F\left(\sigma(n,\cdot)\right) - H\sigma(n,n)\right) v \omega_F.
\]
Now we insert this result into~(\ref{IBP1}) to get
\begin{align*}
\sum_T &\int_T \langle \mathbb{S}\sigma, \nabla\nabla v \rangle \omega - \sum_F \int_F \mathbb{S}\sigma(n,n) \llbracket \nabla_n v \rrbracket \omega_F \\
&= \sum_T \bigg( \sum_{F \subset \partial T} \int_{\partial F} \sigma(n,\nu_F) v \omega_{\partial F} - \sum_{F \subset \partial T}  \int_F \left( \dv_F\left(\sigma(n,\cdot)\right) - H\sigma(n,n)\right) v \omega_F \\  &\quad\quad - \int_{\partial T} (\dv\mathbb{S}\sigma)(n) v \omega_{\partial T} + \int_T (\dv\dv\mathbb{S}\sigma)v \omega \bigg).
\end{align*}
The first term can be re-expressed as a sum over interior $(N-2)$-simplices $S$ using our notation from Lemma~\ref{lemma:angledefectdot}, and the next two terms can be re-expressed in terms of jumps across interior faces $F$.  (Integrals over $(N-2)$-simplices $S \subset \partial\Omega$ and $(N-1)$-simplices $F \subset \partial\Omega$ vanish because $v=0$ on $\partial\Omega$.)  The result is 
\begin{align*}
\sum_T &\int_T \langle \mathbb{S}\sigma, \nabla\nabla v \rangle \omega - \sum_F \int_F \mathbb{S}\sigma(n,n) \llbracket \nabla_n v \rrbracket \omega_F = \mathring{\sum_S}  \int_S \sum_{F \supset S} \llbracket \sigma(n,\tau) \rrbracket_F v \omega_S \\
& - \mathring{\sum_F} \int_F \left\llbracket \dv_F\left(\sigma(n,\cdot)\right) - H\sigma(n,n) +(\dv\mathbb{S}\sigma)(n)  \right\rrbracket v \omega_F + \sum_T \int_T (\dv\dv\mathbb{S}\sigma)v \omega.
\end{align*}
\end{proof}

\begin{remark}
Many of the above calculations are similar to the ones in~\cite[Proposition 4.2]{berchenko2022finite}, except that here we are in dimension $N$ rather than $2$.
\end{remark}

We can now state the main result of this subsection.

\begin{theorem} \label{thm:distcurvdot}
Let $g(t)$ be a family of Regge metrics with time derivative $\frac{\partial}{\partial t}g =: \sigma$. For every $v \in V$, we have
\begin{equation} \label{distcurvdot}
\frac{d}{dt} \langle (R\omega)_{\rm dist}(g(t)), v \rangle_{V',V} 
= b_h(g; \sigma, v) - a_h(g; \sigma, v),
\end{equation}
where
\begin{align*}
b_h(g;\sigma,v) &= \sum_T \int_T \langle \mathbb{S}\sigma, \nabla\nabla v \rangle \omega - \sum_F \int_{F} \mathbb{S}\sigma(n,n) \llbracket \nabla_n v \rrbracket_F \omega_F, \\
a_h(g;\sigma,v) &= \sum_T \int_T \langle G, \sigma \rangle v \omega_T + \mathring{\sum_F} \int_F \left\langle \llbracket \overline{\sff} \rrbracket_F,  \left.\sigma\right|_F \right\rangle v \omega_F - \mathring{\sum_S} \int_S \left\langle \Theta_S g|_S, \sigma|_S \right\rangle v \omega_S.
\end{align*}
\end{theorem}
\begin{proof}
Combine~(\ref{distcurvdot0}) with Lemma~\ref{lemma:IBPdivdivS}.
\end{proof}

\subsection{Distributional densitized Einstein tensor} \label{sec:einstein}

We now pause to make a few remarks about the bilinear forms $a_h(g;\cdot,\cdot)$ and $b_h(g;\cdot,\cdot)$ appearing in Theorem~\ref{thm:distcurvdot}.  These remarks will play no role in our analysis, but they help to elucidate the content of Theorem~\ref{thm:distcurvdot}.  The reader can safely skip ahead to Section~\ref{sec:convergence} if desired.  

Numerical analysts will likely recognize the bilinear form $b_h(g;\cdot,\cdot)$ appearing in Theorem~\ref{thm:distcurvdot}.  As we mentioned in Section~\ref{sec:intro}, it is (up to the appearance of $\mathbb{S}$) a non-Euclidean, $N$-dimensional generalization of a bilinear form that appears in the Hellan-Herrmann-Johnson finite element method~\cite{babuska1980analysis,arnold1985mixed,brezzi1977mixed,braess2018two,braess2019equilibration,arnold2020hellan,pechstein2017tdnns,chen2018multigrid}.  It can be regarded as the integral of $\dv\dv\mathbb{S} \sigma$ against $v$, where $\dv\dv$ is interpreted in a distributional sense.

The bilinear form $a_h(g;\cdot,\cdot)$ can be understood by comparing Theorem~\ref{thm:distcurvdot} with Proposition~\ref{prop:curvdot}, which, when integrated against a continuous function $v$, states that for a family of smooth Riemannian metrics $g(t)$ with scalar curvature $R$,
\begin{equation} \label{intcurvdot}
\frac{d}{dt} \int_\Omega R v \omega = \int_\Omega (\dv\dv\mathbb{S}\sigma) v \omega - \int_\Omega \langle G, \sigma \rangle v \omega,
\end{equation}
where $\sigma = \frac{\partial}{\partial t}g$ and $G=\Ric - \frac{1}{2}Rg$ is the Einstein tensor associated with $g$.  A comparison of~(\ref{intcurvdot}) with~(\ref{distcurvdot}) suggests that for a Regge metric $g$, the bilinear form $a_h(g;\sigma,v)$ should be regarded as a distributional counterpart of $\int_\Omega \langle G, \sigma \rangle v \omega$.

This motivates the following definition.  Fix a number $s >1$, and let $\Sigma$ denote the space of square-integrable symmetric $(0,2)$-tensor fields $\sigma$ with the following properties: the restriction of $\sigma$ to each $T \in \mathcal{T}^N$ belongs to $H^s(T)$, and the tangential-tangential components of $\sigma$ along any face $F \in \mathring{\mathcal{T}}^{N-1}$ are single-valued.  Note that these conditions imply that the tangential-tangential components of $\sigma$ along any $S \in \mathring{\mathcal{T}}^{N-2}$ are well-defined and single-valued as well.
\begin{definition} \label{def:distein}
Let $g$ be a Regge metric.  The \emph{distributional densitized Einstein tensor} associated with $g$ is the linear functional $(G\omega)_{\rm dist}(g) \in \Sigma'$ defined by
\[
\begin{split}
\langle (G\omega)_{\rm dist}(g), \sigma \rangle_{\Sigma',\Sigma}
&=
\sum_T \int_T \langle G, \sigma \rangle \omega_T + \mathring{\sum_F} \int_F \left\langle \llbracket \overline{\sff} \rrbracket_F,  \left.\sigma\right|_F \right\rangle \omega_F - \mathring{\sum_S} \int_S \left\langle \Theta_S g|_S, \sigma|_S \right\rangle \omega_S, \quad \forall \sigma \in \Sigma.
\end{split}
\]
\end{definition}
\begin{remark}
In dimension $N=2$, we have $(G\omega)_{\rm dist}(g)=0$ for any Regge metric $g$, because $G$ vanishes within each triangle, $\bar{\sff}$ vanishes on each edge, and the restriction of $\sigma$ to each vertex vanishes.
\end{remark}
\begin{remark} \label{remark:israel}
The appearance of the trace-reversed second fundamental form $\overline{\sff}$ in Definition~\ref{def:distein} is quite natural.  The same quantity arises in studies of singular sources in general relativity, with the jump in $\overline{\sff}$ encoding the well-known \emph{Israel junction conditions} across a hypersurface on which stress-energy is concentrated~\cite{israel1966singular}.
\end{remark}
\begin{remark}
If we define a map $(\dv\dv\mathbb{S})_{\rm dist} : \Sigma \to V'$ by
\[
\langle (\dv\dv\mathbb{S})_{\rm dist} \sigma, v \rangle_{V',V} = b_h(g;\sigma,v), \quad \forall v \in V,
\]
then, by construction, we have
\[
\left.\frac{d}{dt}\right|_{t=0} \langle (R\omega)_{\rm dist}(g+t\sigma), v \rangle_{V',V} = \langle (\dv\dv\mathbb{S})_{\rm dist} \sigma, v \rangle_{V',V} - \langle (G\omega)_{\rm dist}(g), v\sigma \rangle_{\Sigma',\Sigma}
\]
for every piecewise smooth $\sigma \in \Sigma$ and every smooth function $v$ with compact support in $\Omega$.  In particular, suppose that $\Omega$ has no boundary (e.g., suppose that $\Omega$ is an $N$-dimensional cube and we identify its opposing faces). Then $b_h(g;\sigma,1)=0$ and
\[
\left.\frac{d}{dt}\right|_{t=0} \langle (R\omega)_{\rm dist}(g+t\sigma), 1 \rangle_{V',V} = - \langle (G\omega)_{\rm dist}(g), \sigma \rangle_{\Sigma',\Sigma}
\]
for every piecewise smooth $\sigma \in \Sigma$.  This implies that a Regge metric $g$ is a stationary point of $\langle (R\omega)_{\rm dist}(g), 1 \rangle_{\Sigma',\Sigma}$ if its distributional densitized Einstein tensor vanishes: $(G\omega)_{\rm dist}(g) = 0$.

The functional $\langle (R\omega)_{\rm dist}(g), 1 \rangle_{\Sigma',\Sigma}$ is a counterpart of the Einstein-Hilbert functional $\int_\Omega R\omega$ from general relativity, whose stationary points are solutions to the (vacuum) Einstein field equations $G=0$.  It reduces to the Regge action from Regge calculus when $g$ is piecewise constant.  That is,
\[
\langle (R\omega)_{\rm dist}(g), 1 \rangle_{\Sigma',\Sigma} = 2 \mathring{\sum_S} \Theta_S V_S, \quad \text{ if $g$ is piecewise constant, }
\]
where $V_S = \int_S \omega_S$ denotes the volume of $S$.  If $g$ varies with $t$ and remains piecewise constant for all $t$, then
\[
\frac{d}{dt} 2 \mathring{\sum_S} \Theta_S V_S = 2 \mathring{\sum_S} \dot{\Theta}_S V_S + 2 \mathring{\sum_S} \Theta_S \dot{V}_S,
\]
and one checks that (on a domain without boundary)
\[
2 \mathring{\sum_S} \dot{\Theta}_S V_S = b_h(g;\sigma,1) = 0
\]
and
\[
2 \mathring{\sum_S} \Theta_S \dot{V}_S = -a_h(g;\sigma,1) = -\langle (G\omega)_{\rm dist}(g), \sigma \rangle_{\Sigma',\Sigma},
\]
where $\sigma = \frac{\partial}{\partial t}g$.  The fact that $\mathring{\sum}_S \dot{\Theta}_S V_S=0$ for any piecewise constant Regge metric $g$ (on a domain without boundary) was proved in Regge's classic paper~\cite{regge1961general} using very different techniques.
\end{remark}
\begin{remark}
If $g$ is a Regge metric and $\sigma=gv$ for some smooth function  $v$ with compact support in $\Omega$, then:
\begin{enumerate}
\item On each $N$-simplex $T$, we have
\[
\langle G, \sigma \rangle = \langle G, g \rangle v =  (\Tr G)v = -\left( \frac{N-2}{2} \right) Rv.
\]
\item On either side of each interior $(N-1)$-simplex $F$, we have:
\begin{align*}
\left\langle \overline{\sff} ,  \left.\sigma\right|_F \right\rangle
&= \left\langle \sff,  \left.g\right|_F \right\rangle v - \left\langle \left. g \right|_F,  \left.g\right|_F \right\rangle H v \\
&= Hv - (N-1)Hv \\
&= -(N-2)Hv.
\end{align*}
\item On each interior $(N-2)$-simplex $S$, we have
\[
\left\langle \Theta_S g|_S, \sigma|_S \right\rangle = \Theta_S v \Tr(\left.g\right|_S) = (N-2) \Theta_S v.
\]
\end{enumerate}
This shows that
\begin{align*}
\langle (G\omega)_{\rm dist}(g), gv \rangle_{\Sigma',\Sigma} 
&= -\left(\frac{N-2}{2}\right) \left( \sum_T \int_T R_T v \omega_T + 2\mathring{\sum_F} \int_F \llbracket H \rrbracket_F v \omega_F + 2\mathring{\sum_S} \int_S \Theta_S v \omega_S \right) \\
&= -\left(\frac{N-2}{2}\right) \langle (R\omega)_{\rm dist}(g), v \rangle_{V',V}
\end{align*}
for every smooth function $v$ with compact support in $\Omega$.  
One can interpret this as saying that the trace of $(G\omega)_{\rm dist}(g)$ is $-\left(\frac{N-2}{2}\right)(R\omega)_{\rm dist}(g)$.
\end{remark}
\begin{remark}
If $g$ is a piecewise constant Regge metric and $\sigma \in \Sigma$ is piecewise constant, then
\[
\langle (G\omega)_{\rm dist}(g), \sigma \rangle_{\Sigma',\Sigma}  = -\mathring{\sum_S} \int_S \Theta_S \Tr(\sigma|_S) \omega_S.
\]
If we linearize around the Euclidean metric $g=\delta$, then we see from~(\ref{angledefectdot}) that
\begin{align*}
\left.\frac{d}{dt}\right|_{t=0} \left\langle (G\omega)_{\rm dist}(\delta+t\rho), \sigma \right\rangle_{\Sigma',\Sigma}
&= -\mathring{\sum_S} \int_S \dot{\Theta}_S \Tr(\sigma|_S) \omega_S \\
&= -\frac{1}{2} \mathring{\sum_S} \int_S \sum_{F \supset S} \llbracket \rho(n,\tau) \rrbracket_F \Tr(\sigma|_S) \omega_S
\end{align*}
for every piecewise constant $\rho,\sigma \in \Sigma$.  (Note that there are no additional terms on the right-hand side because $\Theta_S=0$ at $t=0$.)  Hence, if $\Omega$ has no boundary, then
\begin{align*}
\left.\frac{d^2}{dt^2}\right|_{t=0} \langle (R\omega)_{\rm dist}(\delta + t\sigma), 1 \rangle_{V',V}
&= -\left.\frac{d}{dt}\right|_{t=0} \left\langle (G\omega)_{\rm dist}(\delta+t\sigma), \sigma \right\rangle_{\Sigma',\Sigma}  \\
&= \frac{1}{2} \mathring{\sum_S} \int_S \sum_{F \supset S} \llbracket \sigma(n,\tau) \rrbracket_F \Tr(\sigma|_S) \omega_S
\end{align*}
for every piecewise constant $\sigma \in \Sigma$.  This is equivalent to Christiansen's formula~\cite[Theorem 2 and Equations (25-26)]{christiansen2011linearization} for the second variation of the Regge action around the Euclidean metric in dimension $N=3$.  (There, the Regge action is taken to be $\frac{1}{2} \langle (R\omega)_{\rm dist}(g), 1 \rangle_{V',V}$ rather than $\langle (R\omega)_{\rm dist}(g), 1 \rangle_{V',V}$.)

\end{remark}

\section{Convergence} \label{sec:convergence}

In this section, we prove a convergence result for the distributional densitized scalar curvature in the norm
\begin{equation} \label{2norm}
\|u\|_{H^{-2}(\Omega)} = \sup_{\substack{v \in H^2_0(\Omega), \\ v \neq 0}} \frac{ \langle u, v \rangle_{H^{-2}(\Omega),  H^2_0(\Omega)} }{ \|v\|_{H^2(\Omega)} }.
\end{equation}
Our convergence result will be applicable to a family $\{g_h\}_{h>0}$ of Regge metrics defined on a shape-regular family $\{\mathcal{T}_h\}_{h>0}$ of triangulations of $\Omega$ parametrized by $h = \max_{T \in \mathcal{T}_h^N} h_T$, where $h_T = \operatorname{diam}(T)$.   Shape-regularity means that there exists a constant $C_0$ independent of $h$ such that
\[
\max_{T \in \mathcal{T}_h^N} \frac{h_T}{\rho_T} \le C_0
\]
for all $h>0$, where $\rho_T$ denotes the inradius of $T$.  

\begin{theorem} \label{thm:conv}
Let $\Omega \subset \mathbb{R}^N$ be a polyhedral domain equipped with a smooth Riemannian metric $g$.  Let $\{g_h\}_{h>0}$ be a family of Regge metrics defined on a shape-regular family $\{\mathcal{T}_h\}_{h>0}$ of triangulations of $\Omega$.  Assume that $\lim_{h \to 0} \|g_h-g\|_{L^\infty(\Omega)} = 0$ and $C_1 := \sup_{h>0} \max_{T \in \mathcal{T}_h^N} \|g_h\|_{W^{1,\infty}(T)} < \infty$.  The following statements hold:
\begin{enumerate}[label=(\roman*)]
\item \label{thm:conv:part1} If $N=2$, then there exist positive constants $C$ and $h_0$ such that
\begin{equation} \label{convN2}
\begin{split}
\| (R\omega)_{\rm dist}(g_h) - (R\omega)(g)\|_{H^{-2}(\Omega)} &\le C \left( 1 + \max_T  h_T^{-1} \| g_h - g \|_{L^\infty(T)} + \max_T |g_h-g|_{W^{1,\infty}(T)}  \right) \\
&\quad\times \left( \|g_h-g\|_{L^2(\Omega)}^2 + \sum_T h_T^2 |g_h-g|_{H^1(T)}^2 \right)^{1/2}
\end{split}
\end{equation}
for all $h \le h_0$.  The constants $C$ and $h_0$ depend on $\|g\|_{W^{1,\infty}(\Omega)}$, $\|g^{-1}\|_{L^\infty(\Omega)}$, $C_0$, and $C_1$.
\item \label{thm:conv:part2} If $N \ge 3$, assume additionally that $C_2 := \sup_{h>0} \max_{T \in \mathcal{T}_h^N} |g_h|_{W^{2,\infty}(T)} < \infty$.
Then there exist positive constants $C$ and $h_0$ such that
\begin{equation} \label{convN3}
\begin{split}
\| (R\omega)_{\rm dist}(g_h) - (R\omega)(g)&\|_{H^{-2}(\Omega)} \le C \left( 1 + \max_T h_T^{-2} \|g_h-g\|_{L^\infty(T)} + \max_T h_T^{-1} |g_h-g|_{W^{1,\infty}(T)}  \right) \\
&\quad\times \left( \|g_h-g\|_{L^2(\Omega)}^2 + \sum_T h_T^2 |g_h-g|_{H^1(T)}^2 + \sum_T h_T^4 |g_h-g|_{H^2(T)}^2 \right)^{1/2}
\end{split}
\end{equation}
for all $h \le h_0$.  The constants $C$ and $h_0$ depend on $N$, $\|g\|_{W^{1,\infty}(\Omega)}$, $\|g^{-1}\|_{L^\infty(\Omega)}$, $C_0$, $C_1$, and $C_2$.
\end{enumerate}
\end{theorem}

The above theorem leads immediately to error estimates of optimal order for piecewise polynomial interpolants of $g$ having degree $r \ge 0$, provided that either $N = 2$ or $r \ge 1$.  To make this statement precise, we introduce a definition.  Recall that the Regge finite element space of degree $r \ge 0$ consists of symmetric $(0,2)$-tensor fields on $\Omega$ that are piecewise polynomial of degree at most $r$ and possess single-valued tangential-tangential components on interior $(N-1)$-simplices. 

\begin{definition} \label{def:optimalorder}
Let $\mathcal{I}_h$ be a map that sends smooth symmetric $(0,2)$-tensor fields on $\Omega$ to the Regge finite element space of degree $r \ge 0$.   We say that $\mathcal{I}_h$ is an \emph{optimal-order interpolation operator} of degree $r$ if there exists a number $m \in \{0,1,\dots,N\}$ and a constant $C_3=C_3(N,r,h_T/\rho_T,t,s)$ such that for every $p \in [1,\infty]$, every $s \in (m/p,r+1]$, every $t \in [0,s]$, and every symmetric $(0,2)$-tensor field $g$ possessing $W^{s,p}(\Omega)$-regularity, $\mathcal{I}_h g$ exists (upon continuously extending $\mathcal{I}_h$) and satisfies
\begin{equation} \label{optimalorder}
|\mathcal{I}_h g - g|_{W^{t,p}(T)} \le C_3 h_T^{s-t} |g|_{W^{s,p}(T)}
\end{equation}
for every $T \in \mathcal{T}_h^N$.  We call the number $m$ the \emph{codimension index} of $\mathcal{I}_h$.  A Regge metric $g_h$ is called an \emph{optimal-order interpolant} of $g$ having degree $r$ and codimension index $m$ if it is the image of a Riemannian metric $g$ under an  optimal-order interpolation operator having degree $r$ and codimension index $m$.  
\end{definition}

An example of an optimal-order interpolation operator is the canonical interpolation operator onto the degree-$r$ Regge finite element space introduced in~\cite[Chapter 2]{li2018regge}.   Its degrees of freedom involve integrals over simplices of codimension at most $N-1$, so its action on a tensor field $g$ is well-defined so long as $g$ admits traces on simplices of codimension at most $N-1$, i.e.~$g$ possesses $W^{s,p}(\Omega)$-regularity with $s>(N-1)/p$.  Correspondingly, its codimension index is $m=N-1$.

\begin{corollary} \label{cor:conv}
Let $\Omega$, $g$, and $\{\mathcal{T}_h\}_{h>0}$ be as in Theorem~\ref{thm:conv}.  Let $\{g_h\}_{h>0}$ be a family of optimal-order interpolants of $g$ having degree $r \ge 0$ and codimension index $m$. If $N \ge 3$, assume that $r \ge 1$. Then there exist positive constants $C$ and $h_0$ such that
\[
\| (R\omega)_{\rm dist}(g_h) - (R\omega)(g)\|_{H^{-2}(\Omega)} \le C \left( \sum_T h_T^{p(r+1)} |g|_{W^{r+1,p}(T)}^p \right)^{1/p}
\]
for all $h \le h_0$ and all $p \in [2,\infty]$ satisfying $p > \frac{m}{r+1}$.  (We interpret the right-hand side as $C\max_T h_T^{r+1} |g|_{W^{r+1,\infty}(T)}$ if $p=\infty$.)  The constants $C$ and $h_0$ depend on the same quantities listed in~\ref{thm:conv:part1} (if $N=2$) and~\ref{thm:conv:part2} (if $N \ge 3$), as well as on $\Omega$, $r$, and (if $N \ge 3$) $|g|_{W^{2,\infty}(\Omega)}$.
\end{corollary}

\begin{remark} \label{remark:optimalorder_general}
The corollary above continues to hold if we allow slightly more general interpolants in Definition~\ref{def:optimalorder}.  For example, it holds if~(\ref{optimalorder}) is replaced by 
		\begin{equation} \label{optimalorder_general}
		|\mathcal{I}_h g - g|_{W^{t,p}(T)} \le C_3 h_T^{s-t} \sum_{\substack{T' : T' \cap T \neq \emptyset}} |g|_{W^{s,p}(T')},
		\end{equation}
where the sum is over all 	$T' \in \mathcal{T}_h^N$ that share a subsimplex with $T$.
\end{remark}

In what follows, we reuse the letter $C$ to denote a positive constant that may change at each occurrence and may depend on $N$, $\|g\|_{W^{1,\infty}(\Omega)}$, $\|g^{-1}\|_{L^\infty(\Omega)}$, $C_0$, and $C_1$.  Beginning in Lemma~\ref{lemma:einbound}, we allow $C$ to also depend on $C_2$.

Our strategy for proving Theorem~\ref{thm:conv} will be to consider an evolving metric
\[
\gt(t) = (1-t)g + tg_h
\]
with time derivative
\[
\sigma = \frac{\partial}{\partial t}\gt(t) = g_h-g.
\]
Note that $\gt(t)$, being piecewise smooth and tangential-tangential continuous, is a Regge metric for all $t \in [0,1]$, and it happens to be a (globally) smooth Riemannian metric at $t=0$.
Since $\gt(0)=g$ and $\gt(1)=g_h$, Theorem~\ref{thm:distcurvdot} implies that
\[
\langle (R\omega)_{\rm dist}(g_h) - (R\omega)(g), v \rangle_{V',V} =
\int_0^1 b_h(\gt(t);\sigma,v) - a_h(\gt(t);\sigma,v) \, dt, \quad \forall v \in V.
\]
Thus, we can estimate $(R\omega)_{\rm dist}(g_h) - (R\omega)(g)$ by estimating the bilinear forms $b_h(\gt(t);\cdot,\cdot)$ and $a_h(\gt(t);\cdot,\cdot)$.

To do this, we introduce some notation.  Given any Regge metric $g$, we let $\nabla_g$ and $\nabla$ denote the covariant derivatives with respect to $g$ and $\delta$, respectively. Similarly, we append a subscript $g$ to other operators like $\Tr$, $\mathbb{S}$, and $\dv$ when they are taken with respect to $g$, and we omit the subscript when they are taken with respect to $\delta$. On the boundary of any $N$-simplex $T$, we let $n_g$ and $n$ denote the outward unit normal vectors with respect to $\left.g\right|_T$ and $\delta$, respectively.  These two vectors are related to one another in coordinates via
\begin{equation} \label{normal}
n_g = \frac{1}{\sqrt{n^T g^{-1} n}} g^{-1} n, 
\end{equation} 
where we are thinking of $g$ as a matrix and $n$ and $n_g$ as column vectors.  We write $\langle \cdot, \cdot \rangle_g$ for the $g$-inner product of two tensor fields.  If $D$ is a submanifold of $\Omega$ on which the induced metric $\left.g\right|_D$ is well-defined, and if $\rho$ is a tensor field on $D$, then we denote
\[
\|\rho\|_{L^p(D,g)} = 
\begin{cases}
\left( \int_D |\rho|_g^p \, \omega_D(g) \right)^{1/p}, &\mbox{ if } 1 \le p < \infty, \\
\sup_D |\rho|_g, &\mbox{ if } p=\infty,
\end{cases}
\]
where $\omega_D(g)$ is the induced volume form on $D$ and $|\rho|_g = \langle \rho, \rho \rangle_g^{1/2}$.  We abbreviate $\|\cdot\|_{L^p(D)} = \|\cdot\|_{L^p(D,\delta)}$ and $|\cdot|=|\cdot|_\delta$.

We introduce two metric-dependent, mesh-dependent norms.  For $v \in V$, we set 
\[
\|v\|_{2,h,g}^2 = \sum_T \|\nabla_g \nabla_g v\|_{L^2(T,g)}^2 + \sum_F h_F^{-1} \left\| \llbracket dv(n_g) \rrbracket \right\|_{L^2(F,g)}^2.
\]
If $\sigma$ is a symmetric $(0,2)$-tensor field with the property that $\sigma(n_g,n_g)$ is well-defined and single-valued on every $F \in \mathcal{T}_h^{N-1}$, then we set
\[
\|\sigma\|_{0,h,g}^2 = \sum_T \|\sigma\|_{L^2(T,g)}^2 + \sum_F h_F \| \sigma(n_g,n_g)\|_{L^2(F,g)}^2,
\]
where $h_F$ is the Euclidean diameter of $F$.  
Note that the image under $\mathbb{S}_g$ of any symmetric $(0,2)$-tensor field possessing single-valued tangential-tangential components along faces automatically possesses single-valued normal-normal components along faces, because 
\[
\mathbb{S}_g\sigma(n_g,n_g) = \sigma(n_g,n_g) - g(n_g,n_g)\Tr_g\sigma = -\Tr_g\left(\left.\sigma\right|_F\right).
\]

Now we return to the setting of Theorem~\ref{thm:conv} and the discussion thereafter: $g$ is a smooth Riemannian metric, $g_h$ is a Regge metric, $\gt(t) = (1-t)g+tg_h$, and $\sigma = g_h-g$.  We assume throughout what follows that $\lim_{h \to 0} \|g_h-g\|_{L^\infty(\Omega)} = 0$ and $\sup_{h>0} \max_{T \in \mathcal{T}_h^N} \|g_h\|_{W^{1,\infty}(T)} < \infty$.  These assumptions have some elementary consequences that we record here for reference (see~\cite{gawlik2020high} for a derivation).  For every $h$ sufficiently small, every $t \in [0,1]$, and every vector $w$ with unit Euclidean length, 
\begin{align}
\|\gt\|_{L^\infty(\Omega)} + \|\gt^{-1}\|_{L^\infty(\Omega)} &\le C, \label{gtbound} \\
 \max_T |\gt|_{W^{1,\infty}(T)} &\le C,  \\
C^{-1} \le \inf_{\Omega} (w^T \gt w) \le \sup_{\Omega} (w^T \gt w) &\le C, \label{eigbound}
\end{align}
where we are thinking of $\gt$ as a matrix and $w$ as a column vector in the last line.  Note that the last line implies the existence of positive lower and upper bounds on $w^T \gt^{-1} w$ as well:
\begin{equation} \label{eiginvbound}
C^{-1} \le \inf_{\Omega} (w^T \gt^{-1} w) \le \sup_{\Omega} (w^T \gt^{-1} w) \le C.
\end{equation}
In addition, the inequalities $\|\gt\|_{L^\infty(\Omega)} \le C$ and $\|\gt^{-1}\|_{L^\infty(\Omega)} \le C$ imply that
\begin{equation} \label{Lpequiv1}
C^{-1} \|\rho\|_{L^p(D,\gt(t_2))} \le \|\rho\|_{L^p(D,\gt(t_1))} \le C \|\rho\|_{L^p(D,\gt(t_2))}
\end{equation}
and
\begin{equation} \label{Lpequiv2}
C^{-1} \|\rho\|_{L^p(D)} \le \|\rho\|_{L^p(D,\gt(t_1))} \le C \|\rho\|_{L^p(D)}
\end{equation}
for every $t_1,t_2 \in [0,1]$, every admissible submanifold $D$, every $p \in [1,\infty]$, every tensor field $\rho$ having finite $L^p(D)$-norm, and every $h$ sufficiently small.  We select $h_0>0$ so that~(\ref{gtbound}-\ref{Lpequiv2}) hold for all $h \le h_0$, and we tacitly use these inequalities throughout our analysis.

We will show the following near-equivalence of the norms $\|\cdot\|_{2,h,\gt}$ and $\|\cdot\|_{2,h,g}$.

\begin{proposition} \label{prop:equiv}
For every $v \in V$, every $h \le h_0$, and every $t \in [0,1]$,
\[\begin{split}
\|v\|_{2,h,\gt}^2 \le C \bigg[ \|v\|_{2,h,g}^2 + \left( \max_T  h_T^{-2} \| g_h - g \|_{L^\infty(T)}^2 + \max_T |g_h-g|_{W^{1,\infty}(T)}^2 \right) \\ \times \sum_T \left( \|dv\|_{L^2(T)}^2 + h_T^2 |dv|_{H^1(T)}^2 \right) \bigg] .
\end{split}\]
\end{proposition}

The proof of Proposition~\ref{prop:equiv} relies on the following lemma.
\begin{lemma} \label{lemma:ndiff}
Let $\g1$ and $\g2$ be two symmetric positive definite matrices, and let $n$ be a unit vector.  Let 
\[
n_{\g{i}} = \frac{1}{\sqrt{n^T \g{i}^{-1} n}} \g{i}^{-1} n, \quad i=1,2.
\]
Then there exists a constant $c$ depending on $|\g1|, |\g2|, |\g1^{-1}|, |\g2^{-1}|$ such that
\[
|n_{\g1} - n_{\g2}| \le c|\g1-\g2|.
\]
\end{lemma}
\begin{proof}
Using the identity
\begin{equation} \label{sqrtdiff}
\frac{1}{\sqrt{n^T \g1^{-1} n}} - \frac{1}{\sqrt{n^T \g2^{-1} n}}= \frac{ n^T (\g2^{-1}-\g1^{-1}) n }{ n^T \g1^{-1} n \sqrt{n^T \g2^{-1} n} + n^T \g2^{-1} n \sqrt{n^T \g1^{-1} n} },
\end{equation}
we can write
\[
n_{\g1} - n_{\g2} = \frac{ n^T (\g2^{-1}-\g1^{-1}) n }{ n^T \g1^{-1} n \sqrt{n^T \g2^{-1} n} + n^T \g2^{-1} n \sqrt{n^T \g1^{-1} n} } \g1^{-1} n + \frac{1}{\sqrt{n^T \g2^{-1} n}} (\g1^{-1} - \g2^{-1}) n.
\]
Since $\g1^{-1}-\g2^{-1} = \g1^{-1} (\g2 - \g1) \g2^{-1}$, the bound follows easily.
\end{proof}

Notice that in view of~(\ref{normal}), Lemma~\ref{lemma:ndiff} implies that
\begin{equation} \label{ndiffLinf}
\|n_{\gt}-n_g\|_{L^\infty(F)} \le  C\| \gt - g \|_{L^\infty(F)}
\end{equation}
on either side of any face $F$.  

Now we are ready to begin proving Proposition~\ref{prop:equiv}.  Consider the term $\sum_F h_F^{-1} \left\| \llbracket dv(n_{\gt}) \rrbracket \right\|_{L^2(F,\gt)}^2$ that appears in the definition of $\|v\|_{2,h,\gt}^2$.  Notice that
\begin{align*}
dv(n_{\gt}) 
&= dv(n_g) + dv(n_{\gt}-n_g),
\end{align*}
and we can use the bound~(\ref{ndiffLinf}) to estimate
\begin{align*}
\|dv(n_{\gt}-n_g)\|_{L^2(F,\gt)}
&\le C \|dv(n_{\gt}-n_g)\|_{L^2(F)} \\
&\le C \|dv\|_{L^2(F)} \|n_{\gt}-n_g\|_{L^\infty(F)} \\
&\le C \|dv\|_{L^2(F)} \| \gt - g \|_{L^\infty(F)} \\
&\le C \|dv\|_{L^2(F)} \| g_h - g \|_{L^\infty(F)}
\end{align*}
on either side of $F$.  Using the trace inequality 
\begin{equation} \label{traceineq}
\|dv\|_{L^2(F)}^2 \le C\left( h_T^{-1} \|dv\|_{L^2(T)}^2 + h_T |dv|_{H^1(T)}^2 \right), \quad F \subset T \in \mathcal{T}_h^N,
\end{equation}
it follows that
\begin{align*}
\sum_F &h_F^{-1} \| \llbracket dv(n_{\gt}) \rrbracket \|_{L^2(F,\gt)}^2 \\
&\le C \left( \sum_F h_F^{-1} \|\llbracket dv(n_g) \rrbracket \|_{L^2(F,g)}^2 + \sum_T h_T^{-1} \left( h_T^{-1} \|dv\|_{L^2(T)}^2 + h_T |dv|_{H^1(T)}^2 \right) \| g_h - g \|_{L^\infty(T)}^2  \right) \\
&= C \left( \sum_F h_F^{-1} \|\llbracket dv(n_g) \rrbracket \|_{L^2(F,g)}^2 + \sum_T \left( h_T^{-2} \| g_h - g \|_{L^\infty(T)}^2 \|dv\|_{L^2(T)}^2 + \| g_h - g \|_{L^\infty(T)}^2 |dv|_{H^1(T)}^2 \right)  \right),
\end{align*}
where we have used~(\ref{Lpequiv1}),~(\ref{traceineq}), and the bound $h_T \le C h_F$, which follows from the shape-regularity of $\mathcal{T}_h$.

Next, consider the term $\sum_T \| \nabla_{\gt} \nabla_{\gt} v \|_{L^2(T,\gt)}^2$ that appears in the definition of $\|v\|_{2,h,\gt}^2$.  Notice that
\[
\left( \nabla_{\gt} \nabla_{\gt} v \right)_{ij} = \left( \nabla_g \nabla_g v \right)_{ij} + ( \Gamma_{ij}^k - \widetilde{\Gamma}_{ij}^k ) \frac{\partial v}{\partial x^k},
\]
where $\Gamma_{ij}^k$ and $\widetilde{\Gamma}_{ij}^k$ are the Christoffel symbols of the second kind associated with $g$ and $\gt$, respectively.  We have
\[
\|\Gamma_{ij}^k - \widetilde{\Gamma}_{ij}^k\|_{L^\infty(T)} \le C\|\gt-g\|_{W^{1,\infty}(T)} \le C\|g_h-g\|_{W^{1,\infty}(T)},
\]
so
\begin{align*}
\| \nabla_{\gt} \nabla_{\gt} v \|_{L^2(T,\gt)}
&\le C \| \nabla_{\gt} \nabla_{\gt} v \|_{L^2(T)} \\
&\le C \left( \| \nabla_g \nabla_g v \|_{L^2(T)} + \|g_h-g\|_{W^{1,\infty}(T)} \|dv\|_{L^2(T)} \right) \\
&\le C \left( \| \nabla_g \nabla_g v \|_{L^2(T,g)} + \|g_h-g\|_{W^{1,\infty}(T)} \|dv\|_{L^2(T)} \right).
\end{align*}
It follows that
\[\begin{split}
\|v\|_{2,h,\gt}^2 \le C \bigg[ \|v\|_{2,h,g}^2 + \left( \max_T h_T^{-2} \| g_h - g \|_{L^\infty(T)}^2 + \max_T |g_h-g|_{W^{1,\infty}(T)}^2 \right) \\ \times \sum_T \left( \|dv\|_{L^2(T)}^2 + h_T^2 |dv|_{H^1(T)}^2 \right) \bigg].
\end{split}\]
This completes the proof of Proposition~\ref{prop:equiv}.

Our next step will be to estimate the bilinear form $b_h(\gt;\cdot,\cdot)$.
\begin{proposition} \label{prop:bhbound}
For every $h \le h_0$, every $t \in [0,1]$, and every $v \in H^2_0(\Omega)$, we have (with $\sigma=g_h-g$)
\[
\begin{split}
|b_h(\gt;\sigma,v)| &\le C \left( \|g_h-g\|_{L^2(\Omega)}^2 + \sum_T h_T^2 |g_h-g|_{H^1(T)}^2 \right)^{1/2} \\
&\quad\times \left( 1 + \max_T h_T^{-1} \| g_h - g \|_{L^\infty(T)} + \max_T |g_h-g|_{W^{1,\infty}(T)}  \right) \|v\|_{H^2(\Omega)}.
\end{split}
\]
\end{proposition}

\begin{proof}
In view of the definitions of $\|\cdot\|_{0,h,\gt}$ and $\|\cdot\|_{2,h,\gt}$, we have
\begin{equation} \label{bhbounded}
|b_h(\gt;\sigma,v)| \le \|\mathbb{S}_{\gt} \sigma\|_{0,h,\gt} \|v\|_{2,h,\gt}.
\end{equation}
Recalling that
\[
\|\mathbb{S}_{\gt} \sigma\|_{0,h,\gt}^2 = \sum_T \|\mathbb{S}_{\gt}\sigma\|_{L^2(T,\gt)}^2 + \sum_F h_F \|\mathbb{S}_{\gt}\sigma(n_{\gt},n_{\gt})\|_{L^2(F,\gt)}^2,
\]
we compute
\begin{align*}
\langle \mathbb{S}_{\gt}\sigma, \mathbb{S}_{\gt}\sigma \rangle_{\gt}
&= \Big\langle \sigma - \gt \langle \gt, \sigma \rangle_{\gt}, \, \sigma - \gt \langle \gt,\sigma \rangle_{\gt} \Big\rangle_{\gt} \\
&= \langle \sigma, \sigma \rangle_{\gt} - 2 \langle \gt, \sigma \rangle_{\gt}^2 + \langle \gt, \gt \rangle_{\gt} \langle \gt, \sigma \rangle_{\gt}^2 \\
&= \langle \sigma, \sigma \rangle_{\gt} + (N-2) \langle \gt, \sigma \rangle_{\gt}^2,
\end{align*}
which leads to the bound
\[
\|\mathbb{S}_{\gt}\sigma\|_{L^2(T,\gt)}
\le C\|\sigma\|_{L^2(T,\gt)} \le C \|\sigma\|_{L^2(T)}.
\]
Also, by the trace inequality,
\begin{align*}
\|\mathbb{S}_{\gt}\sigma(n_{\gt},n_{\gt})\|_{L^2(\partial T,\gt)}^2
&\le C \|\mathbb{S}_{\gt}\sigma\|_{L^2(\partial T,\gt)}^2 \\
&\le C \|\sigma\|_{L^2(\partial T)}^2 \\
&\le C \left( h_T^{-1}\|\sigma\|_{L^2(T)}^2 + h_T |\sigma|_{H^1(T)}^2 \right).
\end{align*}
(Here we are measuring the $L^2(\partial T,\gt)$-norm of the full tensor $\mathbb{S}_{\gt}\sigma$ rather than its restriction to the tangent bundle of $\partial T$.)
Thus,
\begin{align}
\|\mathbb{S}_{\gt} \sigma\|_{0,h,\gt}^2 
&\le C \left( \|\sigma\|_{L^2(\Omega)}^2 + \sum_T h_T^2 |\sigma|_{H^1(T)}^2 \right) \nonumber \\
&= C \left( \|g_h-g\|_{L^2(\Omega)}^2 + \sum_T h_T^2 |g_h-g|_{H^1(T)}^2 \right). \label{normSgsigma}
\end{align}
Consider now the term $\|v\|_{2,h,\gt}$ in~(\ref{bhbounded}).  Proposition~\ref{prop:equiv} implies that
\[
\|v\|_{2,h,\gt} \le C\left( \|v\|_{2,h,g} + \left( \max_T h_T^{-1} \| g_h - g \|_{L^\infty(T)} + \max_T |g_h-g|_{W^{1,\infty}(T)} \right) \|v\|_{H^2(\Omega)} \right)
\]
since $v \in H^2_0(\Omega)$.  Furthermore, since $g$ is smooth and $v \in H^2_0(\Omega)$, we have $\llbracket dv(n_g) \rrbracket=0$ on every interior face $F$ and $\llbracket dv(n_g) \rrbracket=dv(n_g)=0$ on every face $F \subset \partial\Omega$.  Thus, $ \|v\|_{2,h,g}^2 = \sum_T \|\nabla_g \nabla_g v \|_{L^2(T,g)}^2 = \|\nabla_g \nabla_g v \|_{L^2(\Omega,g)}^2$.  Since 
\[
(\nabla_g \nabla_g v)_{ij} = (\nabla \nabla v)_{ij} - \Gamma^k_{ij} \frac{\partial v}{\partial x^k},
\]
we see that
\[
\|v\|_{2,h,g} = \|\nabla_g \nabla_g v\|_{L^2(\Omega)} \le C( |v|_{H^2(\Omega)} + |v|_{H^1(\Omega)} ) \le C \|v\|_{H^2(\Omega)}.
\]
Thus,
\begin{equation} \label{normv}
\|v\|_{2,h,\gt} \le C\left( 1 + \max_T h_T^{-1} \| g_h - g \|_{L^\infty(T)} + \max_T |g_h-g|_{W^{1,\infty}(T)} \right) \|v\|_{H^2(\Omega)}.
\end{equation}
Combining~(\ref{bhbounded}),~(\ref{normSgsigma}), and~(\ref{normv}) completes the proof.
\end{proof}

At this point, we have finished proving part~\ref{thm:conv:part1} of Theorem~\ref{thm:conv}.  Indeed, in dimension $N=2$, $a_h$ vanishes, so we can write
\begin{align*}
\left| \langle (R\omega)_{\rm dist}(g_h) - (R\omega)(g), v \rangle_{V',V} \right|
&\le \int_0^1 |b_h(\gt(t);\sigma,v)| \, dt
\end{align*}
and apply Proposition~\ref{prop:bhbound} to deduce~(\ref{convN2}). 

To prove part~\ref{thm:conv:part2} of Theorem~\ref{thm:conv}, we suppose that $N \ge 3$ and that $\sup_{h>0} \max_{T \in \mathcal{T}_h^N} |g_h|_{W^{2,\infty}(T)} < \infty$, and we proceed as follows.  Recall that
\begin{equation} \label{ahgtilde}
a_h(\gt;\sigma,v) = \sum_T \int_T \langle G(\gt), \sigma \rangle_{\gt} v \omega_T(\gt) + \mathring{\sum_F} \int_F \left\langle \llbracket \overline{\sff}(\gt) \rrbracket_F,  \left.\sigma\right|_F \right\rangle_{\gt} v \omega_F(\gt) - \mathring{\sum_S} \int_S \langle \Theta_S(\gt) \gt|_S, \sigma|_S \rangle_{\gt} v \omega_S(\gt),
\end{equation}
where have made all dependencies on the metric explicit in the notation.  We will bound each of the three terms above, beginning with the first.  Throughout what follows, we continue to denote $\sigma=g_h-g$, and we let $v$ be an arbitrary member of $V$.

\begin{lemma} \label{lemma:einbound}
We have 
\[
\left| \sum_T \int_T \langle G(\gt), \sigma \rangle_{\gt} \, v \omega_T(\gt) \right| \le C \|g_h-g\|_{L^2(\Omega)} \|v\|_{L^2(\Omega)}.
\]
\end{lemma}
\begin{proof}
Since we are now assuming that $\sup_{h>0} \max_{T \in \mathcal{T}_h^N} \|g_h\|_{W^{2,\infty}(T)} < \infty$, the Einstein tensor associated with $\gt$ satisfies
\[
\|G(\gt)\|_{L^\infty(T)} \le C
\]
for every $h \le h_0$, every $t \in [0,1]$, and every $T \in \mathcal{T}_h^N$.  It follows that
\begin{align*}
\left| \int_T \langle G(\gt), \sigma \rangle_{\gt} \, v \omega_T(\gt) \right|
&\le \|G(\gt)\|_{L^\infty(T,\gt)} \|\sigma\|_{L^2(T,\gt)} \|v\|_{L^2(T,\gt)} \\
&\le C \|G(\gt)\|_{L^\infty(T)} \|\sigma\|_{L^2(T)} \|v\|_{L^2(T)} \\
&\le C \|\sigma\|_{L^2(T)} \|v\|_{L^2(T)} \\
&= C \|g_h-g\|_{L^2(T)} \|v\|_{L^2(T)}.
\end{align*}
Summing over all $T \in \mathcal{T}_h^N$ completes the proof.
\end{proof}

\begin{lemma} \label{lemma:sffbound}
We have
\begin{align*}
&\left| \mathring{\sum_F} \int_F \left\langle \llbracket \overline{\sff}(\gt) \rrbracket_F,  \left.\sigma\right|_F \right\rangle_{\gt} v\omega_F(\gt) \right|
\le C\max_T \left( h_T^{-1} \|g_h-g\|_{W^{1,\infty}(T)}\right) \\&\quad\times \left( \sum_T \|g_h-g\|_{L^2(T)}^2 + h_T^2 |g_h-g|_{H^1(T)}^2 \right)^{1/2} \left( \sum_T \|v\|_{L^2(T)}^2 + h_T^2 |v|_{H^1(T)}^2 \right)^{1/2}.
\end{align*}
\end{lemma}
\begin{proof}
Consider an interior $(N-1)$-simplex $F$.  By applying a Euclidean rotation and translation to the coordinates, we may assume without loss of generality that $F$ lies in the plane $x^N=0$.  In these coordinates, the second fundamental form associated with $\gt$ is given by
\begin{align*}
\sff_{ij}(\gt)
&= -\gt(n_{\gt}, \nabla_{\gt,e_i} e_j) \\
&= -\gt(n_{\gt}, \widetilde{\Gamma}^k_{ij} e_k ) \\
&= -n_{\gt}^\ell \gt_{\ell k} \widetilde{\Gamma}^k_{ij}, \quad i,j=1,2,\dots,N-1,
\end{align*}
where $e_1,e_2,\dots,e_N$ are the Euclidean coordinate basis vectors.  Since $n_{\gt} = \gt^{-1} n / \sqrt{n^T \gt^{-1} n}$ and $n$ points in the $x^N$ direction, we get
\[
\sff_{ij}(\gt) = -\frac{1}{\sqrt{n^T \gt^{-1} n}} \widetilde{\Gamma}^N_{ij}.
\]
The jump in this quantity across $F$ can be computed using the identity $\llbracket ab \rrbracket = \llbracket a \rrbracket \{b\} + \{a\} \llbracket b \rrbracket$, where $\{\cdot\}$ denotes the average across $F$, giving
\[
-\llbracket \sff_{ij}(\gt) \rrbracket = \left\llbracket \frac{1}{\sqrt{n^T \gt^{-1} n}} \right\rrbracket \left\{ \widetilde{\Gamma}^N_{ij} \right\} + \left\{ \frac{1}{\sqrt{n^T \gt^{-1} n}} \right\} \left\llbracket \widetilde{\Gamma}^N_{ij} \right\rrbracket.
\]
In view of~(\ref{sqrtdiff}), we have
\begin{align*}
\left\| \left\llbracket \frac{1}{\sqrt{n^T \gt^{-1} n}} \right\rrbracket \right\|_{L^\infty(F)} 
&\le C \left\| \llbracket \gt \rrbracket \right\|_{L^\infty(F)} \\
&\le C \left\| \llbracket g_h-g \rrbracket \right\|_{L^\infty(F)} \\
&\le C \left( \|g_h-g\|_{L^\infty(T_1)} + \|g_h-g\|_{L^\infty(T_2)} \right),
\end{align*}
where $T_1$ and $T_2$ are the two $N$-simplices that share the face $F$.  Here, we used the fact that $\gt = g + t(g_h-g)$ and $g$ is smooth.  Similarly, we have
\begin{align}
\left\| \left\llbracket \widetilde{\Gamma}^N_{ij} \right\rrbracket \right\|_{L^\infty(F)} 
&\le C \| \llbracket \gt \rrbracket \|_{W^{1,\infty}(F)} \nonumber \\
&\le C \| \llbracket g_h-g \rrbracket \|_{W^{1,\infty}(F)} \nonumber \\
&\le C \left( \|g_h-g\|_{W^{1,\infty}(T_1)} + \|g_h-g\|_{W^{1,\infty}(T_2)} \right). \label{jumpGamma}
\end{align}
Thus,
\[
\| \llbracket \sff(\gt) \rrbracket \|_{L^\infty(F)} \le  C \left( \|g_h-g\|_{W^{1,\infty}(T_1)} + \|g_h-g\|_{W^{1,\infty}(T_2)} \right).
\]
From this it follows easily that the same bound holds, possibly with a larger constant $C$, for the trace-reversed tensor $\overline{\sff}(\gt)=\sff(\gt)-H(\gt)\gt$:
\[
\| \llbracket \overline{\sff}(\gt) \rrbracket \|_{L^\infty(F)} \le  C \left( \|g_h-g\|_{W^{1,\infty}(T_1)} + \|g_h-g\|_{W^{1,\infty}(T_2)} \right).
\]
It follows that
\begin{align*}
&\left| \int_F \left\langle \llbracket \overline{\sff}(\gt) \rrbracket_F,  \left.\sigma\right|_F \right\rangle_{\gt} v\omega_F(\gt) \right| \\
&\le \|\llbracket \overline{\sff}(\gt) \rrbracket \|_{L^\infty(F,\gt)} \|\sigma|_F \|_{L^2(F,\gt)} \|v\|_{L^2(F,\gt)} \\
&\le C \|\llbracket \overline{\sff}(\gt) \rrbracket \|_{L^\infty(F)} \|\sigma|_F \|_{L^2(F)} \|v\|_{L^2(F)} \\
&\le  C \left( \sum_{i=1}^2 \|g_h-g\|_{W^{1,\infty}(T_i)} \right)  \left( h_{T_1}^{-1} \|\sigma\|_{L^2(T_1)}^2 + h_{T_1} |\sigma|_{H^1(T_1)}^2 \right)^{1/2} \left( h_{T_1}^{-1} \|v\|_{L^2(T_1)}^2 + h_{T_1} |v|_{H^1(T_1)}^2 \right)^{1/2}.
\end{align*}
By the shape-regularity of $\mathcal{T}_h$, we have $C^{-1} \le h_{T_1}/h_{T_2} \le C$ for some constant $C$ independent of $h$ and $F$, so
\begin{align*}
&\left| \mathring{\sum_F} \int_F \left\langle \llbracket \overline{\sff}(\gt) \rrbracket_F,  \left.\sigma\right|_F \right\rangle_{\gt} v\omega_F(\gt) \right| 
\le C\max_T \left( h_T^{-1} \|g_h-g\|_{W^{1,\infty}(T)}\right) \\&\quad\times \left( \sum_T \|g_h-g\|_{L^2(T)}^2 + h_T^2 |g_h-g|_{H^1(T)}^2 \right)^{1/2} \left( \sum_T \|v\|_{L^2(T)}^2 + h_T^2 |v|_{H^1(T)}^2 \right)^{1/2}.
\end{align*}

\end{proof}

\begin{remark} \label{remark:sffbound}
If $g_h$ is piecewise constant, then in~(\ref{jumpGamma}) we have the sharper bound
\[
\| \llbracket g_h-g \rrbracket \|_{W^{1,\infty}(F)} = \| \llbracket g_h-g \rrbracket \|_{L^\infty(F)} \le C \left( \|g_h-g\|_{L^\infty(T_1)} + \|g_h-g\|_{L^\infty(T_2)} \right)
\]
because $\frac{\partial g_h}{\partial x^i}=0$ and $\frac{\partial g}{\partial x^i}$ is continuous for each $i$.  This implies that for piecewise constant $g_h$, we can replace $\|g_h-g\|_{W^{1,\infty}(T)}$ by $\|g_h-g\|_{L^\infty(T)}$ in Lemma~\ref{lemma:sffbound}, yielding
\begin{align*}
&\left| \mathring{\sum_F} \int_F \left\langle \llbracket \overline{\sff}(\gt) \rrbracket_F,  \left.\sigma\right|_F \right\rangle_{\gt} v\omega_F(\gt) \right| \le C\max_T \left( h_T^{-1} \|g_h-g\|_{L^\infty(T)}\right) \\&\quad\times \left( \sum_T \|g_h-g\|_{L^2(T)}^2 + h_T^2 |g_h-g|_{H^1(T)}^2 \right)^{1/2} \left( \sum_T \|v\|_{L^2(T)}^2 + h_T^2 |v|_{H^1(T)}^2 \right)^{1/2}.
\end{align*}
\end{remark}

Now we turn our attention toward the third integral in~(\ref{ahgtilde}). In preparation for this, we will first use the shape-regularity assumption to show that the dihedral angles of every $N$-simplex in $\mathcal{T}_h$ (measured in the Euclidean metric) are uniformly bounded above and below.
\begin{lemma} \label{lemma:anglebound}
There exist constants $\theta_{\rm min}, \theta_{\rm max} \in (0,\pi)$ such that for every $h>0$ and every $T \in \mathcal{T}_h^N$, the dihedral angles in $T$ (measured in the Euclidean metric) all lie between $\theta_{\rm min}$ and $\theta_{\rm max}$.  
\end{lemma}
\begin{proof}
This fact is proved in dimension $N=3$ in~\cite[Lemma 3.6]{gong2022note}.  We generalize their proof to dimension $N \ge 3$ as follows.  Given $N+1$ points $x_1,x_2,\dots,x_{N+1}$ in general position in $\mathbb{R}^N$, let $T=[x_1,x_2,\dots,x_{N+1}]$ denote the $N$-simplex with vertices $x_1,x_2,\dots,x_{N+1}$.  Consider two faces $F_1=[x_1,x_3,x_4,\dots,x_{N+1}]$ and $F_2=[x_2,x_3,x_4,\dots,x_{N+1}]$ that intersect along the $(N-2)$-dimensional subsimplex $S=[x_3,x_4,\dots,x_{N+1}]$.  Throughout what follows, we work in the Euclidean metric.  Let $A$ be the orthogonal projection of $x_1$ onto the $(N-1)$-dimensional hyperplane containing $F_2$, and let $B$ be the orthogonal projection of $x_1$ onto the $(N-2)$-dimensional hyperplane containing $S$.  Observe that both $[x_1,A]$ and $[x_1,B]$ are orthogonal to $S$, since $S \subset F_2$.  Thus, the triangle $[x_1,A,B]$ is orthogonal to $S$.  This triangle is a right triangle with hypotenuse $[x_1,B]$, so the dihedral angle $\theta_{ST}$ along $S$ satisfies
\[
\sin\theta_{ST} = \frac{|[x_1,A]|}{|[x_1,B]|},
\]
where $|\cdot|$ denotes the Euclidean volume (i.e.~length in this case).  Obviously, $|[x_1,B]|$ is bounded above by $h_T$, the diameter of $T$.  In addition, $|[x_1,A]|$ is bounded from below by 2 times $\rho_T$, the inradius of $T$.  To see why, we generalize the argument in~\cite[Proposition 2.3]{gong2022note}, bearing in mind that our definition of $\rho_T$ differs from theirs by a factor of $2$.  Consider the inscribed $(N-1)$-sphere in $T$, whose center $C$ lies at a distance $\rho_T$ from $F_2$.  Let $D$ be the point where this inscribed sphere touches $F_2$, and let $E$ be the point diametrically opposite to $D$ on this sphere.  The line segment $[D,E]$ is orthogonal to $F_2$, so the volume of the $N$-simplex $T'=[E,x_2,x_3,x_4,\dots,x_{N+1}]$ satisfies
\[
|T'| = \frac{1}{N} |[D,E]| |F_2| = \frac{2\rho_T}{N} |F_2|.
\]
Since $T' \subset T$, we have 
\[
|T'| \le |T| = \frac{1}{N} |[x_1,A]| |F_2|,
\]
so
\[
2\rho_T \le |[x_1,A]|.
\]
Thus,
\[
\sin\theta_{ST} \ge \frac{2\rho_T}{h_T}.
\]
The result follows from this bound and the shape-regularity of $\mathcal{T}_h$.
\end{proof}

Next we show that Lemma~\ref{lemma:anglebound} remains valid when one measures angles with $g$ rather than the Euclidean metric $\delta$.

\begin{lemma}
Upon reducing the value of $h_0$ if necessary, there exist constants $\theta_{{\rm min},g}, \theta_{{\rm max},g} \in (0,\pi)$ such that for every $h \le h_0$, every $T \in \mathcal{T}_h^N$, every $(N-2)$-simplex $S \subset \partial T$, and every point $p \in S$, the dihedral angle in $T$ at $p$ (measured by $g$) lies between $\theta_{{\rm min},g}$ and $\theta_{{\rm max},g}$.  
\end{lemma}
\begin{proof}
If there were no such lower bound $\theta_{{\rm min},g}>0$, then there would exist a sequence of $N$-simplices $T_1 \in T_{h_1}$, $T_2 \in \mathcal{T}_{h_2}$, $\dots$ with faces $F_1^{(1)}, F_1^{(2)} \subset T_1$, $F_2^{(1)}, F_2^{(2)} \subset T_2$, $\dots$ and points $p_1 \in F_1^{(1)} \cap F_1^{(2)} ,p_2 \in F_2^{(1)} \cap F_2^{(2)},\dots$ such that
\[
\angle_{\left.g\right|_{T_i}(p_i)}(F_i^{(1)},F_i^{(2)}) \to 0
\]
as $i \to \infty$, where $\angle_g(X,Y)$ denotes the angle between $X$ and $Y$ as measured by $g$.  Using the compactness of the Grassmannian, this implies that, after extracting a subsequence which we do not relabel,
\[
\angle_{\delta}(F_i^{(1)},F_i^{(2)}) \to 0,
\]
where $\angle_\delta(X,Y)$ denotes the angle between $X$ and $Y$ as measured by the Euclidean metric $\delta$. This contradicts the assumed positive lower bound on the Euclidean dihedral angles.  The existence of an upper bound $\theta_{{\rm max},g}<\pi$ is proved similarly.
\end{proof}

Now we are ready to estimate the third integral in~(\ref{ahgtilde}).

\begin{lemma} \label{lemma:angledefectbound}
We have
\begin{align*}
&\left| \mathring{\sum_S} \int_S \left\langle \Theta_S(\gt) \left.\gt\right|_S, \left.\sigma\right|_S \right\rangle_{\gt} v \omega_S(\gt) \right| \\
&\le C \left( \max_T h_T^{-2} \|g_h-g\|_{L^\infty(T)} \right) \left( \sum_T \|g_h-g\|_{L^2(T)}^2 + h_T^2 |g_h-g|_{H^1(T)}^2 + h_T^4 |g_h-g|_{H^2(T)}^2 \right)^{1/2} \\
&\quad\times \left( \sum_T \|v\|_{L^2(T)}^2 + h_T^2 |v|_{H^1(T)}^2 + h_T^4 |v|_{H^2(T)}^2 \right)^{1/2}.
\end{align*}
\end{lemma}

\begin{proof}
Fix an interior $(N-2)$-simplex $S$ and an $N$-simplex $T$ containing $S$.  At any point $p$ along $S$, we have
\begin{align*}
\cos\theta_{ST}(g)-\cos\theta_{ST}(\gt)
&= \gt(n_{\gt}^{(1)}, n_{\gt}^{(2)}) - g(n_g^{(1)}, n_g^{(2)}) \\
&= \gt( n_{\gt}^{(1)} - n_g^{(1)}, n_{\gt}^{(2)} - n_g^{(2)} ) + \gt( n_{\gt}^{(1)} - n_g^{(1)}, n_g^{(2)} ) + \gt( n_g^{(1)}, n_{\gt}^{(2)} - n_g^{(2)} ) \\
&\quad + \gt(n_g^{(1)},n_g^{(2)}) - g(n_g^{(1)},n_g^{(2)}),
\end{align*}
where $n_g^{(1)}$ and $n_g^{(2)}$ are suitably oriented unit normal vectors (with respect to $\left.g\right|_T$) to the two faces of $T$ containing $S$, and similarly for  $n_{\gt}^{(1)}$ and $n_{\gt}^{(2)}$.  Using Lemma~\ref{lemma:ndiff}, we see that at the point $p$,
\[
|\cos\theta_{ST}(\gt)-\cos\theta_{ST}(g)| \le C|\gt-g| \le C|g_h-g|
\]
for all $h$ sufficiently small.  Since there are constants $\theta_{{\rm min},g},\theta_{{\rm max},g} \in (0,\pi)$ such that $\theta_{{\rm min},g} \le \theta_{ST}(g) \le \theta_{{\rm max},g}$, we get
\[
|\theta_{ST}(\gt)-\theta_{ST}(g)| \le C|g_h-g| \le C\|g_h-g\|_{L^\infty(T)}.
\]
Summing over $T \supset S$ and noting that $\sum_{T \supset S} \theta_{ST}(g) = 2\pi$, we get
\begin{equation} \label{ThetaSbound}
|\Theta_S(\gt)| = |\Theta_S(\gt)-\Theta_S(g)| \le \sum_{T \supset S} |\theta_{ST}(\gt)-\theta_{ST}(g)| \le C \sum_{T \supset S} \|g_h-g\|_{L^\infty(T)}.
\end{equation}
Now we are almost ready to estimate the integral $\int_S \left\langle \Theta_S(\gt) \left.\gt\right|_S, \left.\sigma\right|_S \right\rangle_{\gt} v \omega_S(\gt)$.  We first note that
\[
\|v\|_{L^2(S)}^2 \le C\left( h_{T}^{-2} \|v\|_{L^2(T)}^2 + |v|_{H^1(T)}^2 + h_{T}^2 |v|_{H^2(T)}^2 \right),
\]
which can be proved using a codimension-2 trace inequality and a scaling argument, or by applying the codimension-1 trace inequality~(\ref{traceineq}) twice (to $v$ rather than $dv$).  If $T_1,T_2,\dots,T_m$ are the $N$-simplices that share the $(N-2)$-simplex $S$, then we have
\begin{align*}
&\left| \int_S \left\langle \Theta_S(\gt) \left.\gt\right|_S, \left.\sigma\right|_S \right\rangle_{\gt} v \omega_S(\gt) \right| \\
&\le C \|\Theta_S(\gt)\|_{L^\infty(S,\gt)} \|\sigma|_S\|_{L^2(S,\gt)} \|v\|_{L^2(S,\gt)} \\
&\le C \|\Theta_S(\gt)\|_{L^\infty(S)} \|\sigma|_S\|_{L^2(S)} \|v\|_{L^2(S)} \\
&\le C \left( \sum_{i=1}^m \|g_h-g\|_{L^\infty(T_i)} \right) \left( h_{T_1}^{-2} \|\sigma\|_{L^2(T_1)}^2 + |\sigma|_{H^1(T_1)}^2 + h_{T_1}^2 |\sigma|_{H^2(T_1)}^2 \right)^{1/2} \\
&\quad\times \left( h_{T_1}^{-2} \|v\|_{L^2(T_1)}^2 + |v|_{H^1(T_1)}^2 + h_{T_1}^2 |v|_{H^2(T_1)}^2 \right)^{1/2}.
\end{align*}
The proof is completed by summing over all interior $(N-2)$-simplices $S$ and substituting $\sigma=g_h-g$.
\end{proof}

Collecting our results, we can state a bound on the bilinear form $a_h(\gt;\cdot,\cdot)$.
\begin{proposition} \label{prop:ahbound}
For every $h \le h_0$, every $t \in [0,1]$, and every $v \in V$, we have (with $\sigma = g_h-g$),
\begin{align*}
|a_h(\gt; \sigma, v)| 
&\le C \left( 1 + \max_T h_T^{-2} \|g_h-g\|_{L^\infty(T)} + \max_T h_T^{-1} |g_h-g|_{W^{1,\infty}(T)} \right) \\
&\quad\times \left( \sum_T \|g_h-g\|_{L^2(T)}^2 + h_T^2 |g_h-g|_{H^1(T)}^2 + h_T^4 |g_h-g|_{H^2(T)}^2 \right)^{1/2} \\
&\quad\times\left( \sum_T \|v\|_{L^2(T)}^2 + h_T^2 |v|_{H^1(T)}^2 + h_T^4 |v|_{H^2(T)}^2 \right)^{1/2}.
\end{align*}
\end{proposition}
\begin{proof}
Combine Lemmas~\ref{lemma:einbound},~\ref{lemma:sffbound}, and~\ref{lemma:angledefectbound}.
\end{proof}
Upon combining Proposition~\ref{prop:bhbound} with Proposition~\ref{prop:ahbound}, we see that
\[
\begin{split}
\| (R\omega)_{\rm dist}(g_h) - (R\omega)(g)\|_{H^{-2}(\Omega)} &\le C \left( 1 + \max_T h_T^{-2} \|g_h-g\|_{L^\infty(T)} + \max_T h_T^{-1} |g_h-g|_{W^{1,\infty}(T)}  \right) \\
&\quad\times \left( \|g_h-g\|_{L^2(\Omega)}^2 + \sum_T h_T^2 |g_h-g|_{H^1(T)}^2 + \sum_T h_T^4 |g_h-g|_{H^2(T)}^2 \right)^{1/2}.
\end{split}
\]
This completes the proof of Theorem~\ref{thm:conv}.  Corollary~\ref{cor:conv} then follows from~(\ref{optimalorder}) and the bounds 
\begin{align*}
\|g_h-g\|_{L^2(\Omega)} &\le |\Omega|^{1/2-1/p} \|g_h-g\|_{L^p(\Omega)}, \\
\left( \sum_T h_T^2 |g_h-g|_{H^1(T)}^2 \right)^{1/2} &\le |\Omega|^{1/2-1/p} \left( \sum_T h_T^p |g_h-g|_{W^{1,p}(T)}^p \right)^{1/p}, \\
\left( \sum_T h_T^4 |g_h-g|_{H^2(T)}^2 \right)^{1/2} &\le |\Omega|^{1/2-1/p} \left( \sum_T h_T^{2p} |g_h-g|_{W^{2,p}(T)}^p \right)^{1/p},
\end{align*}
which hold for all $p \in [2,\infty]$ (with the obvious modifications for $p=\infty$).

\begin{remark} \label{remark:suboptimal}
Notice that the analysis above yields
\begin{align}
|b_h(\gt;\sigma,v)| &= O(h^{r+1}), && \text{(by Proposition~\ref{prop:bhbound})}, \\
\left| \sum_T \int_T \langle G(\gt), \sigma \rangle_{\gt} \, v \omega_T(\gt) \right| &= O(h^{r+1}), && \text{(by Lemma~\ref{lemma:einbound})}, \label{einO} \\
\left| \mathring{\sum_F} \int_F \left\langle \llbracket \overline{\sff}(\gt) \rrbracket_F,  \left.\sigma\right|_F \right\rangle_{\gt} v\omega_F(\gt) \right| &= \begin{cases} 
O(h), &\mbox{ if }  r = 0,\\
O(h^{2r}), &\mbox{ if } r \ge 1,
\end{cases}
&&\begin{aligned}
&\text{(by Remark~\ref{remark:sffbound})}, \\
&\text{(by Lemma~\ref{lemma:sffbound})},
\end{aligned} \\
\left| \mathring{\sum_S} \int_S \left\langle \Theta_S(\gt) \left.\gt\right|_S, \left.\sigma\right|_S \right\rangle_{\gt} v \omega_S(\gt) \right| &= O(h^{2r}), && \text{(by Lemma~\ref{lemma:angledefectbound})} \label{angledefectO}
\end{align}
for any optimal-order interpolant $g_h$ of $g$ having degree $r \ge 0$.  Bearing in mind that~(\ref{einO}-\ref{angledefectO}) vanish when $N=2$, we see that the above estimates lead to an optimal error estimate $\| (R\omega)_{\rm dist}(g_h) - (R\omega)(g)\|_{H^{-2}(\Omega)} = O(h^{r+1})$ in all cases except when $N \ge 3$ and $r=0$, where we obtain $\| (R\omega)_{\rm dist}(g_h) - (R\omega)(g)\|_{H^{-2}(\Omega)} = O(1)$ because of~(\ref{angledefectO}).  Numerical experiments suggest that these analytical results are sharp for a general optimal-order interpolant, whereas for the canonical interpolant the estimate~(\ref{angledefectO}) improves to $O(h^{2(r+1)})$, yielding $\| (R\omega)_{\rm dist}(g_h) - (R\omega)(g)\|_{H^{-2}(\Omega)} = O(h)$ when $r=0$; cf. Figure~\ref{fig:conv_plot_crit_term}.
\end{remark}

\section{Numerical examples} \label{sec:numerical}
In this section we present numerical experiments in dimension $N=2,3$ to illustrate the predicted convergence rates. The examples were performed in the open source finite element library NGSolve\footnote{www.ngsolve.org} \cite{schoeberl1997netgen,schoeberl2014ngsolve}, where the Regge finite elements are available for arbitrary polynomial order. We construct an optimal-order interpolant $g_h$ of a given metric tensor $g$ as follows. On each element $T$, the local $L^2$ best-approximation $\bar{g}_h|_T$ of $g|_T$ is computed. Then the tangential-tangential degrees of freedom shared by two or more neighboring elements are averaged to obtain a globally tangential-tangential continuous interpolant $g_h$.  We verify in Appendix~\ref{appendix:optimalorder} that this interpolant is an optimal-order interpolant in the sense of Remark~\ref{remark:optimalorder_general} on shape-regular, quasi-uniform triangulations.

To compute the $H^{-2}(\Omega)$-norm of the error $f:= (R\omega)_{\rm dist}(g_h)-(R\omega)(g)$ we make use of the fact that $\|f\|_{H^{-2}(\Omega)}$ is equivalent to $\|u\|_{H^2(\Omega)}$, where $u\in H^2_0(\Omega)$ solves the biharmonic equation $\Delta^2 u = f$. This equation will be solved numerically using the (Euclidean) Hellan--Herrmann--Johnson method. To prevent the discretization error from spoiling the real error, we use for $u_h$ two polynomial orders more than for $g_h$.

We consider in dimension $N=2$ the numerical example proposed in \cite{gawlik2020high}, where on the square $\Omega=(-1,1)^2$ the smooth Riemannian metric tensor
\begin{align*}
g(x,y):= \begin{pmatrix}
1+(\frac{\partial f}{\partial x})^2 & \frac{\partial f}{\partial x}\frac{\partial f}{\partial y} \\ \frac{\partial f}{\partial x}\frac{\partial f}{\partial y} & 1+ (\frac{\partial f}{\partial y})^2\end{pmatrix}
\end{align*}
with $f(x,y):= \frac{1}{2}(x^2+y^2)-\frac{1}{12}(x^4+y^4)$ is defined. This metric corresponds to the surface induced by the embedding $\big(x,y\big)\mapsto \big(x,y,f(x,y)\big)$, and its exact scalar curvature is given by
\begin{align*}
R(g)(x,y) = \frac{162(1-x^2)(1-y^2)}{(9+x^2(x^2-3)^2+y^2(y^2-3)^2)^2}.
\end{align*}

For a three-dimensional example we consider the cube $\Omega=(-1,1)^3$ and the Riemannian metric tensor induced by the embedding $\big(x,y,z\big)\mapsto \big(x,y,z,f(x,y,z)\big)$, where $f(x,y,z):=\frac{1}{2}(x^2+y^2+z^2)-\frac{1}{12}(x^4+y^4+z^4)$. The scalar curvature is 
\[
R(g)(x,y,z) =\frac{18\left( (1-x^2)(1-y^2)(9+q(z)) + (1-y^2)(1-z^2)(9+q(x)) + (1-z^2)(1-x^2)(9+q(y)) \right) }{(9+q(x)+q(y)+q(z))^2},
\]
where $q(x) = x^2 (x^2-3)^2$.

We start with a structured mesh consisting of $2\cdot 2^{2k}$ triangles  and $6\cdot 2^{3k}$ tetrahedra, respectively, in two and three dimensions with $\tilde{h}=\max_T h_T=\sqrt{N}\,2^{1-k}$ (and minimal edge length $2^{1-k}$) for $k=0,1,\dots$. To avoid possible superconvergence due to mesh symmetries, we perturb each component of the inner mesh vertices by a random number drawn from a uniform distribution in the range $[-\tilde{h}\,2^{-(2N+1)/2},\tilde{h}\,2^{-(2N+1)/2}]$. As depicted in Figure~\ref{fig:conv_plot} (left) and listed in Table~\ref{tab:error_N2}, linear convergence is observed when $N=2$ and $g_h$ has polynomial degree $r=0$. This is consistent with Theorem~\ref{thm:conv}\ref{thm:conv:part1}. For $r=1$ and $r=2$, higher convergence rates are obtained as expected.

In the three-dimensional case, the same convergence rates as for $N=2$ are obtained, cf. Figure~\ref{fig:conv_plot} (right) and Table~\ref{tab:error_N3}. This indicates that Theorem~\ref{thm:conv}\ref{thm:conv:part2} is sharp for $r\geq 1$. For $r=0$ we observe numerically linear convergence, which is better than predicted by Theorem~\ref{thm:conv}\ref{thm:conv:part2}. However, further investigation suggests that the observed linear convergence for $r=0$ is pre-asymptotic. Indeed, to test if \eqref{angledefectO} is sharp, we compute the $H^{-2}(\Omega)$-norm of the linear functional
\begin{align}
v \mapsto \int_0^1 \mathring{\sum_S} \int_S \left\langle \Theta_S(\gt(t)) \left.\gt(t)\right|_S, \left.\sigma\right|_S \right\rangle_{\gt(t)} v \omega_S(\gt(t))\,dt, \label{critical_term}
\end{align}
where we approximate the parameter integral by a Gauss quadrature of order seven. As depicted in Figure~\ref{fig:conv_plot_crit_term}, the norm of this functional for the optimal-order interpolant $g_h$ with $r=0$ stagnates at about $4\cdot 10^{-4}$, which is below the overall error of $4.296\cdot 10^{-3}$ for the finest grid; cf. Table~\ref{tab:error_N3}. Therefore, the lack of convergence predicted by Theorem~\ref{thm:conv}\ref{thm:conv:part2} is not yet visible in Figure~\ref{fig:conv_plot}. For $r=1,2$ the proven rate of $O(h^{2r})$ for \eqref{critical_term} (see \eqref{angledefectO}) is clearly obtained. Interestingly, using the canonical interpolant appears to increase the convergence rate of \eqref{critical_term} to $O(h^{2(r+1)})$ (i.e. an increase of two orders), as observed in Figure~\ref{fig:conv_plot_crit_term}.  Thus, it appears that the canonical interpolant achieves convergence in the lowest-order case.  We intend to study this superconvergence phenomenon exhibited by the canonical interpolant in future work.

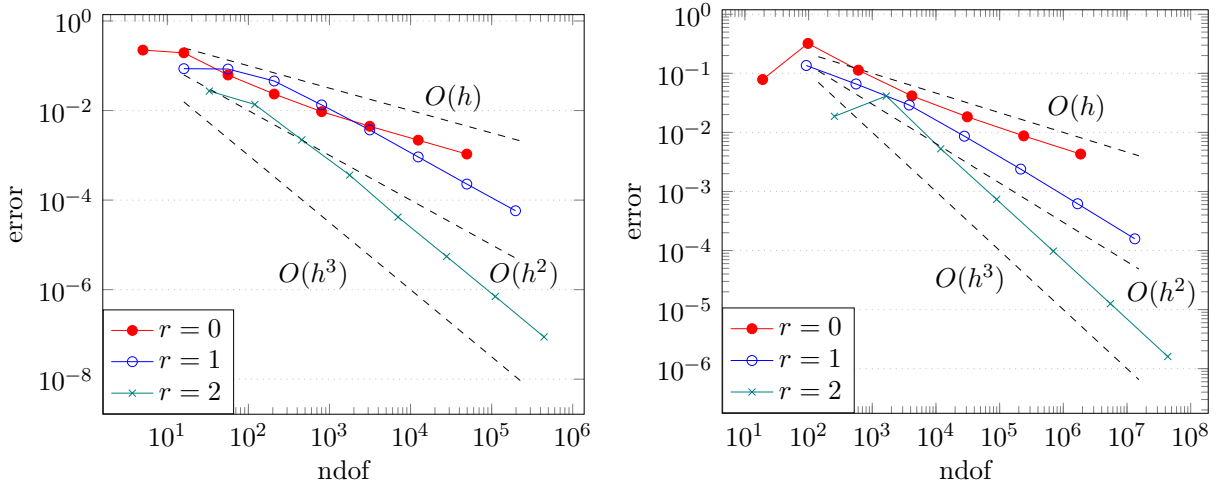
\begin{figure}
	\centering
	\resizebox{0.49\textwidth}{!}{
		\begin{tikzpicture}
		\begin{loglogaxis}[
		legend style={at={(0,0)}, anchor=south west},
		xlabel={ndof},
		ylabel={error},
		ymajorgrids=true,
		grid style=dotted,
		]
		\addlegendentry{$r=0$}
		\addplot[color=red, mark=*] coordinates {
			( 5 ,  0.223745609409278 )
			( 16 ,  0.194462134504585 )
			( 56 ,  0.0621977870431615 )
			( 208 ,  0.0233638487874428 )
			( 800 ,  0.0094339284346834 )
			( 3136 ,  0.0044568730457745 )
			( 12416 ,  0.0021810609037358 )
			( 49408 ,  0.0010666783169232 )
		};
		\addlegendentry{$r=1$}
		\addplot[color=blue, mark=o] coordinates {
			( 16 ,  0.0861267946606119 )
			( 56 ,  0.084479618514948 )
			( 208 ,  0.0456470361522919 )
			( 800 ,  0.0133457830026765 )
			( 3136 ,  0.0036888789020029 )
			( 12416 ,  0.0009204557403272 )
			( 49408 ,  0.0002279974531332 )
			( 197120 ,  5.777307258843193e-05 )
		};
		
		\addlegendentry{$r=2$}
		\addplot[color=teal, mark=x] coordinates {
			( 33 ,  0.0271990483397802 )
			( 120 ,  0.0136406434789509 )
			( 456 ,  0.0022130473620008 )
			( 1776 ,  0.000361512807993 )
			( 7008 ,  4.188545674354635e-05 )
			( 27840 ,  5.503884594546939e-06 )
			( 110976 ,  7.027906861115703e-07 )
			( 443136 ,  8.784421232101793e-08 )
			
		};
	
		\addplot[color=black, mark=none, style=dashed] coordinates {
			( 16, {16^(-1/2)} )
			( 44, {44^(-1/2)} )
			( 220, {220^(-1/2)} )
			( 896, {896^(-1/2)} )
			( 3724, {3724^(-1/2)} )
			( 14624, {14624^(-1/2)} )
			( 57292, {57292^(-1/2)} )
			( 228584, {228584^(-1/2)} )
		};

		\addplot[color=black, mark=none, style=dashed] coordinates {
			( 16, {16^(-2/2)} )
			( 44, {44^(-2/2)} )
			( 220, {220^(-2/2)} )
			( 896, {896^(-2/2)} )
			( 3724, {3724^(-2/2)} )
			( 14624, {14624^(-2/2)} )
			( 57292, {57292^(-2/2)} )
			( 228584, {228584^(-2/2)} )
		};
		\addplot[color=black, mark=none, style=dashed] coordinates {
			( 16, {16^(-3/2)} )
			( 44, {44^(-3/2)} )
			( 220, {220^(-3/2)} )
			( 896, {896^(-3/2)} )
			( 3724, {3724^(-3/2)} )
			( 14624, {14624^(-3/2)} )
			( 57292, {57292^(-3/2)} )
			( 228584, {228584^(-3/2)} )
		};
		
		\end{loglogaxis}
		
		\node (A) at (5, 4.5) [] {$O(h)$};
		\node (B) at (6, 2.) [] {$O(h^2)$};
		\node (C) at (3, 2) [] {$O(h^3)$};
		\end{tikzpicture}}
	\resizebox{0.49\textwidth}{!}{
	\begin{tikzpicture}
	\begin{loglogaxis}[
	legend style={at={(0,0)}, anchor=south west},
	xlabel={ndof},
	ylabel={error},
	ymajorgrids=true,
	grid style=dotted,
	]
	\addlegendentry{$r=0$}
	\addplot[color=red, mark=*] coordinates {
		( 19 ,  0.078689467681022 )
		( 98 ,  0.3215206552300362 )
		( 604 ,  0.1131756821590536 )
		( 4184 ,  0.0415165494783523 )
		( 31024 ,  0.018378173088411 )
		( 238688 ,  0.0087333143327151 )
		( 1872064 ,  0.0042962474180309 )
	};
	\addlegendentry{$r=1$}
	\addplot[color=blue, mark=o] coordinates {
		( 92 ,  0.1358695203463132 )
		( 556 ,  0.0661338606992382 )
		( 3800 ,  0.0291152964208913 )
		( 27952 ,  0.0086328797972605 )
		( 214112 ,  0.0023913193646172 )
		( 1675456 ,  0.0006194122507064 )
		( 13255040 ,  0.000157885729318 )
	};
	
	\addlegendentry{$r=2$}
	\addplot[color=teal, mark=x] coordinates {
		( 255 ,  0.018711639691359 )
		( 1662 ,  0.0413278709438671 )
		( 11892 ,  0.0052855822502922 )
		( 89736 ,  0.0007341956477066 )
		( 696720 ,  9.752865630809082e-05 )
		( 5489952 ,  1.2610044493205028e-05 )
		( 43586112 ,  1.6039007874768194e-06 )
	};
	
	\addplot[color=black, mark=none, style=dashed] coordinates {
		( 142, {142^(-1/3)} )
		( 354, {354^(-1/3)} )
		( 4114, {4114^(-1/3)} )
		( 35698, {35698^(-1/3)} )
		( 245011, {245011^(-1/3)} )
		( 1912096, {1912096^(-1/3)} )
		( 15366973, {15366973^(-1/3)} )
	};
	
	\addplot[color=black, mark=none, style=dashed] coordinates {
		( 142, {3*142^(-2/3)} )
		( 354, {3*354^(-2/3)} )
		( 4114, {3*4114^(-2/3)} )
		( 35698, {3*35698^(-2/3)} )
		( 245011, {3*245011^(-2/3)} )
		( 1912096, {3*1912096^(-2/3)} )
		( 15366973, {3*15366973^(-2/3)} )
	};
	\addplot[color=black, mark=none, style=dashed] coordinates {
		( 142, {10*142^(-3/3)} )
		( 354, {10*354^(-3/3)} )
		( 4114, {10*4114^(-3/3)} )
		( 35698, {10*35698^(-3/3)} )
		( 245011, {10*245011^(-3/3)} )
		( 1912096, {10*1912096^(-3/3)} )
		( 15366973, {10*15366973^(-3/3)} )
	};
	
	\end{loglogaxis}
	
	\node (A) at (5, 4.3) [] {$O(h)$};
	\node (B) at (6.2, 1.7) [] {$O(h^2)$};
	\node (C) at (3.5, 1.9) [] {$O(h^3)$};
	\end{tikzpicture}}
	
	\caption{Convergence of the distributional scalar curvature in the $H^{-2}(\Omega)$-norm for $N=2$ (left) and $N=3$ (right) with respect to the number of degrees of freedom (ndof) of $g_h$ for $r=0,1,2$.}
	\label{fig:conv_plot}
\end{figure}

\begin{figure}
	\centering
	\resizebox{0.49\textwidth}{!}{
		\begin{tikzpicture}
			\begin{loglogaxis}[
				legend style={at={(0,0)}, anchor=south west},
				xlabel={ndof},
				ylabel={$H^{-2}(\Omega)$-norm of \eqref{critical_term}},
				ymajorgrids=true,
				grid style=dotted,
				]
				\addlegendentry{$r=0$}
				\addplot[color=red, mark=*] coordinates {
					(  19 ,  0.0333352303853497  )
					(  98 ,  0.0151071647590974  )
					(  604 ,  0.0016606255389163  )
					(  4184 ,  0.0004320600792596  )
					(  31024 ,  0.0002777136715641  )
					(  238688 ,  0.0003064128874271  )
					(  1872064 ,  0.0003318183264289  )
				};
				\addlegendentry{$r=1$}
				\addplot[color=blue, mark=o] coordinates {
					(  92 ,  0.0253751525541026  )
					(  556 ,  0.000919730328497  )
					(  3800 ,  0.0005292047655846  )
					(  27952 ,  0.0001530060324196  )
					(  214112 ,  4.081533978686004e-05  )
					(  1675456 ,  1.048173633509593e-05  )
					(  13255040 ,  2.635676015094756e-06  )
				};
				
				\addlegendentry{$r=2$}
				\addplot[color=teal, mark=x] coordinates {
					(  255 ,  0.0007779423425818  )
					(  1662 ,  0.0008620305429432  )
					(  11892 ,  6.760475031527034e-06  )
					(  89736 ,  2.2108635825020797e-07  )
					(  696720 ,  1.8613374747783105e-08  )
					(  5489952 ,  1.2087696116457698e-09  )

				};
			
				\addlegendentry{$r=0$ c.i.}
				\addplot[color=orange, mark=square, style=dashed, mark options=solid] coordinates {
					(  19 ,  0.0002362591917287  )
					(  98 ,  0.0098556091258469  )
					(  604 ,  0.0021242374115492  )
					(  4184 ,  0.0005069524972355  )
					(  31024 ,  0.000131172995997  )
					(  238688 ,  3.2987363048662945e-05  )
					(  1872064 ,  8.247188370376352e-06  )
				};
			
				\addlegendentry{$r=1$ c.i.}
				\addplot[color=cyan, mark=square*, style=dashed, mark options=solid] coordinates {
					(  92 ,  0.000358229935323  )
					(  556 ,  0.0002899642860931  )
					(  3800 ,  9.023880253405389e-05  )
					(  27952 ,  6.891132744464625e-06  )
					(  214112 ,  4.68078580597443e-07  )
					(  1675456 ,  3.0358400843173734e-08  )
					(  13255040 ,  1.925729420114136e-09  )
				};
				
				\addlegendentry{$r=2$ c.i.}
				\addplot[color=green, mark=diamond*, style=dashed, mark options=solid] coordinates {
					(  1662 ,  9.838993561202204e-05  )
					(  11892 ,  8.086946075405945e-07  )
					(  89736 ,  1.3912756614505817e-08  )
					(  696720 ,  2.4394797955317406e-10  )
					(  5489952 ,  4.019135797521764e-12  )
				};

				\addplot[color=black, mark=none, style=dashed] coordinates {
					( 142, {142^(-2/3)} )
					( 354, {354^(-2/3)} )
					( 4114, {4114^(-2/3)} )
					( 35698, {35698^(-2/3)} )
					( 245011, {245011^(-2/3)} )
					( 1912096, {1912096^(-2/3)} )
					( 15366973, {15366973^(-2/3)} )
				};
				\addplot[color=black, mark=none, style=dashed] coordinates {
					( 142, {20*142^(-4/3)} )
					( 354, {20*354^(-4/3)} )
					( 4114, {20*4114^(-4/3)} )
					( 35698, {20*35698^(-4/3)} )
					( 245011, {20*245011^(-4/3)} )
					( 1912096, {20*1912096^(-4/3)} )
					( 15366973, {20*15366973^(-4/3)} )
				};
			
				\addplot[color=black, mark=none, style=dashed] coordinates {
					( 142, {20*142^(-6/3)} )
					( 354, {20*354^(-6/3)} )
					( 4114, {20*4114^(-6/3)} )
					( 35698, {20*35698^(-6/3)} )
					( 245011, {20*245011^(-6/3)} )
					( 1912096, {20*1912096^(-6/3)} )
				};
				
			\end{loglogaxis}
			
			\node (A) at (4, 4.7) [] {$O(h^2)$};
			\node (B) at (6.2, 2.6) [] {$O(h^4)$};
			\node (C) at (3.4, 1.8) [] {$O(h^6)$};
	\end{tikzpicture}}
	
	\caption{Convergence of \eqref{critical_term} in the $H^{-2}(\Omega)$-norm with respect to number of degrees of freedom (ndof) for an optimal-order interpolant and the canonical interpolant (c.i.) for $r=0,1,2$ in dimension $N=3$.}
	\label{fig:conv_plot_crit_term}
\end{figure}
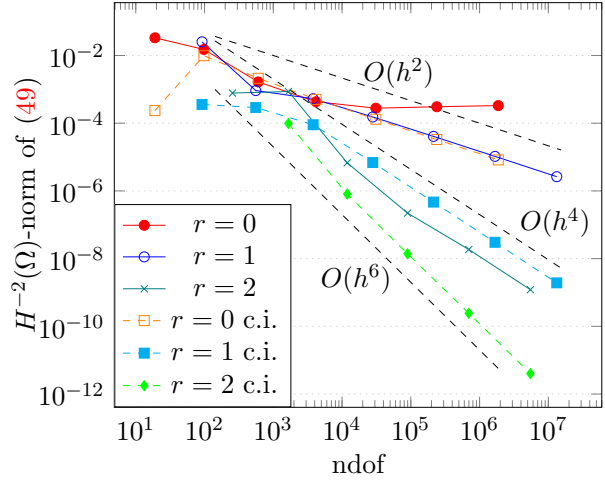

\begin{table}
	\centering
	\begin{tabular}{cccc}
		& $r=0$ & $r=1$ & $r=2$\\
		\hline
		$h$ & \begin{tabular}{@{}ll@{}}
			Error &\hspace{0.2in} Order \\
		\end{tabular} & \begin{tabular}{@{}ll@{}}
			Error &\hspace{0.2in} Order \\
		\end{tabular} & \begin{tabular}{@{}ll@{}}
			Error &\hspace{0.2in} Order \\
		\end{tabular} \\
		\hline
		\begin{tabular}{@{}l@{}}
			$2.828\cdot 10^{-0}$\\
			$1.534\cdot 10^{-0}$\\
			$8.584\cdot 10^{-1}$\\
			$4.609\cdot 10^{-1}$\\
			$2.417\cdot 10^{-1}$\\
			$1.251\cdot 10^{-1}$\\
			$6.260\cdot 10^{-2}$\\
			$3.198\cdot 10^{-2}$
		\end{tabular} & \begin{tabular}{@{}ll@{}}
			$2.237\cdot 10^{-1}$ &  \\
			$1.945\cdot 10^{-1}$ & 0.23\\
			$6.220\cdot 10^{-2}$ & 1.96\\
			$2.336\cdot 10^{-2}$ & 1.57\\
			$9.434\cdot 10^{-3}$ & 1.41\\
			$4.457\cdot 10^{-3}$ & 1.14\\
			$2.181\cdot 10^{-3}$ & 1.03\\
			$1.067\cdot 10^{-3}$ & 1.06
		\end{tabular} &\begin{tabular}{@{}ll@{}}
			$8.613\cdot 10^{-2}$ &  \\
			$8.448\cdot 10^{-2}$ & 0.03\\
			$4.565\cdot 10^{-2}$ & 1.06\\
			$1.335\cdot 10^{-2}$ & 1.98\\
			$3.689\cdot 10^{-3}$ & 1.99\\
			$9.205\cdot 10^{-4}$ & 2.11\\
			$2.280\cdot 10^{-4}$ & 2.02\\
			$5.777\cdot 10^{-5}$ & 2.04
		\end{tabular}&\begin{tabular}{@{}ll@{}}
			$2.720\cdot 10^{-2}$ &  \\
			$1.364\cdot 10^{-2}$ & 1.13\\
			$2.213\cdot 10^{-3}$ & 3.13\\
			$3.615\cdot 10^{-4}$ & 2.91\\
			$4.189\cdot 10^{-5}$ & 3.34\\
			$5.504\cdot 10^{-6}$ & 3.08\\
			$7.028\cdot 10^{-7}$ & 2.97\\
			$8.784\cdot 10^{-8}$ & 3.1
		\end{tabular}
	\end{tabular}
	\caption{Same as Figure~\ref{fig:conv_plot} (left), but in tabular form.}
	\label{tab:error_N2}
\end{table}

\begin{table}
	\centering
	\begin{tabular}{cccc}
		& $r=0$ & $r=1$ & $r=2$\\
		\hline
		$h$ & \begin{tabular}{@{}ll@{}}
			Error &\hspace{0.2in} Order \\
		\end{tabular} & \begin{tabular}{@{}ll@{}}
			Error &\hspace{0.2in} Order \\
		\end{tabular} & \begin{tabular}{@{}ll@{}}
			Error &\hspace{0.2in} Order \\
		\end{tabular} \\
		\hline
		\begin{tabular}{@{}l@{}}
			$3.464\cdot 10^{-0}$\\
			$1.850\cdot 10^{-0}$\\
			$9.709\cdot 10^{-1}$\\
			$4.999\cdot 10^{-1}$\\
			$2.753\cdot 10^{-1}$\\
			$1.358\cdot 10^{-1}$\\
			$6.878\cdot 10^{-2}$
		\end{tabular} &
		\begin{tabular}{@{}ll@{}}
			$7.869\cdot 10^{-2}$ &  \\
			$3.215\cdot 10^{-1}$ & -2.24\\
			$1.132\cdot 10^{-1}$ & 1.62\\
			$4.152\cdot 10^{-2}$ & 1.51\\
			$1.838\cdot 10^{-2}$ & 1.37\\
			$8.733\cdot 10^{-3}$ & 1.05\\
			$4.296\cdot 10^{-3}$ & 1.04
		\end{tabular}
		&\begin{tabular}{@{}ll@{}}
			$1.359\cdot 10^{-1}$ &  \\
			$6.613\cdot 10^{-2}$ & 1.15\\
			$2.912\cdot 10^{-2}$ & 1.27\\
			$8.633\cdot 10^{-3}$ & 1.83\\
			$2.391\cdot 10^{-3}$ & 2.15\\
			$6.194\cdot 10^{-4}$ & 1.91\\
			$1.579\cdot 10^{-4}$ & 2.01
		\end{tabular}
		&\begin{tabular}{@{}ll@{}}
			$1.871\cdot 10^{-2}$ &  \\
			$4.133\cdot 10^{-2}$ & -1.26\\
			$5.286\cdot 10^{-3}$ & 3.19\\
			$7.342\cdot 10^{-4}$ & 2.97\\
			$9.753\cdot 10^{-5}$ & 3.38\\
			$1.261\cdot 10^{-5}$ & 2.89\\
			$1.604\cdot 10^{-6}$ & 3.03
		\end{tabular}
	\end{tabular}
	\caption{Same as Figure~\ref{fig:conv_plot} (right), but in tabular form.}
	\label{tab:error_N3}
\end{table}

\section*{Acknowledgments}
We thank Yasha Berchenko-Kogan for many helpful discussions, especially about the mean curvature term in Definition~\ref{def:distcurv}.  We also thank Snorre Christiansen for pointing out the link with the Israel formalism mentioned in Remark~\ref{remark:israel}.   EG was supported by NSF grant DMS-2012427. MN acknowledges support by the Austrian Science Fund (FWF) project F\,65.

\appendix

\section{Optimal-order interpolation via averaging} \label{appendix:optimalorder}

Below we verify that the interpolant described in Section~\ref{sec:numerical} is an optimal-order interpolant in the sense of Remark~\ref{remark:optimalorder_general}, assuming that $\{\mathcal{T}_h\}_{h>0}$ is shape-regular and quasi-uniform.  Recall that quasi-uniformity means that $\max_{T\in \mathcal{T}_h^N} h/h_T$ is bounded above by a constant independent of $h$.  In what follows, the letter $C$ may depend on this constant as well as on the parameters $N$, $h_T/\rho_T$, $r$, $s$, and $t$ appearing below.

Let $\ell^{(1)},\ell^{(2)},\dots,\ell^{(M)}$ denote the canonical degrees of freedom for the Regge finite element space of degree $r \ge 0$ on $\mathcal{T}_h$~\cite[Equation (2.4b)]{li2018regge}.  Each linear functional $\ell^{(i)}$ is associated with a simplex $D \in \mathcal{T}_h^k$ of dimension $k \ge 1$ in the following sense: $\ell^{(i)}$ sends a symmetric $(0,2)$-tensor field $g$ to the integral of $\left.g\right|_D$ against a (symmetric tensor-valued) polynomial of degree $\le r-k+1$ over $D$.

We enumerate these degrees of freedom with a local numbering system as follows.  On a given $N$-simplex $T \in \mathcal{T}_h^N$,  the degrees of freedom associated with subsimplices of $T$ are denoted $\ell_1^T,\ell_2^T,\dots,\ell_{M_T}^T$.  If $T,T' \in \mathcal{T}_h^N$ are two $N$-simplices with nonempty intersection, then it may happen that $\ell_i^T$ and $\ell_j^{T'}$ coincide for some and $i$ and $j$.  We let $\mathcal{S}(i,T)$ denote the set of all pairs $(j,T')$ for which $\ell_i^T$ and $\ell_j^{T'}$ coincide.

With the above local numbering system, let $\psi_1^T, \psi_2^T, \dots, \psi_{M_T}^T$ denote the basis for the degree-$r$ Regge finite element space that is dual to the above degrees of freedom.  That is,
\[
\ell_i^T(\psi_j^{T'}) = 
\begin{cases} 
	1, &\mbox { if } (j,T') \in \mathcal{S}(i,T), \\
	0, &\mbox{ otherwise. }
\end{cases}
\]
Let us assume that the degrees of freedom and basis functions above are first defined on a reference simplex and then transported to $T$ via an affine transformation.  A scaling argument shows that~\cite[Lemma 2.11]{li2018regge}
\begin{equation} \label{scaling1}
\|\psi_i^T\|_{L^p(T)} \le C h_T^{N/p-2}
\end{equation}
and
\begin{equation} \label{scaling2}
|\ell_i^T(g)| \le C h_T^{-N/p+2} \|g\|_{L^p(T)}
\end{equation}
for all $g$ in the domain of $\ell_i^T$.  Note that the $-2$ and the $+2$ appearing in the exponents above arise because of the way that pullbacks of $(0,2)$-tensor fields behave under affine transformations; see~\cite[Lemma 2.11]{li2018regge}.

Let $g$ be a symmetric $(0,2)$-tensor field possessing $W^{s,p}(\Omega)$-regularity for every $p \in [1,\infty]$ and every $s > (N-1)/p$.  
The canonical interpolation operator $\mathcal{J}_h$ onto the Regge finite element space is defined elementwise by
\[
\left. \mathcal{J}_h g \right|_T = \mathcal{J}_h^T\left( \left.g\right|_T \right)  = \sum_{i=1}^{M_T} \ell_i^T(g) \psi_i^T.
\]

Let $\bar{g}_h$ denote the elementwise $L^2$-projection of $g$ onto the space of discontinuous piecewise polynomial symmetric $(0,2)$-tensor fields of degree at most $r$.  Since $\mathcal{J}_h$ is a projector, we have
\[
\left.\bar{g}_h\right|_T = \mathcal{J}_h^T \left(\left.\bar{g}_h\right|_T \right) = \sum_{i=1}^{M_T} \ell_i^T(\bar{g}_h) \psi_i^T.
\]
The interpolant discussed in Section~\ref{sec:numerical} is defined by
\[
\left.g_h\right|_T = \sum_{i=1}^{M_T} \left( \frac{1}{|\mathcal{S}(i,T)|} \sum_{(j,T') \in \mathcal{S}(i,T)} \ell_j^{T'}(\bar{g}_h) \right) \psi_i^T,
\]
where $|\mathcal{S}(i,T)|$ denotes the cardinality of $\mathcal{S}(i,T)$.

To analyze the error $g_h-g$, let $p \in [1,\infty]$, $s \in ((N-1)/p,r+1]$, and $t \in [0,s]$. We have
\begin{align*}
|g_h-g|_{W^{t,p}(T)} 
&\le |g_h-\mathcal{J}_hg|_{W^{t,p}(T)} + |\mathcal{J}_hg-g|_{W^{t,p}(T)}.
\end{align*}
The second term satisfies~\cite[Theorem 2.5]{li2018regge}
\begin{equation} \label{errorterm1}
 |\mathcal{J}_hg-g|_{W^{t,p}(T)} \le C h_T^{s-t} |g|_{W^{s,p}(T)}.
\end{equation}
To bound the first term, we use the fact that
\[
\ell_i^T(g) = \frac{1}{|\mathcal{S}(i,T)|} \sum_{(j,T') \in \mathcal{S}(i,T)} \ell_j^{T'}(g)
\]
to write
\begin{align*}
\left. ( g_h - \mathcal{J}_hg )\right|_T 
&= \sum_{i=1}^{M_T} \frac{1}{|\mathcal{S}(i,T)|} \sum_{(j,T') \in \mathcal{S}(i,T)} \ell_j^{T'}(\bar{g}_h-g) \psi_i^T.
\end{align*}
Using an inverse estimate,~(\ref{scaling1}),~(\ref{scaling2}), and a standard error estimate~\cite[Proposition 1.135]{ern2004theory} for the elementwise $L^2$-projector, we obtain
\begin{align}
|g_h - \mathcal{J}_h g|_{W^{t,p}(T)} 
&\le C h_T^{-t} \|g_h - \mathcal{J}_h g\|_{L^p(T)} \nonumber \\
&\le C h_T^{-t} \sum_{T' : T' \cap T \neq \emptyset} h_{T'}^{-N/p+2} \|\bar{g}_h-g\|_{L^p(T')} h_T^{N/p-2} \nonumber \\
&\le C h_T^{-t}  \sum_{T' : T' \cap T \neq \emptyset} \|\bar{g}_h-g\|_{L^p(T')} \nonumber \\
&\le C h_T^{-t} \sum_{T' : T' \cap T \neq \emptyset} h_{T'}^s |g|_{W^{s,p}(T')} \nonumber \\
&\le C h_T^{s-t} \sum_{T' : T' \cap T \neq \emptyset} |g|_{W^{s,p}(T')}. \label{errorterm2}
\end{align}
Here, we have repeatedly used the fact that the ratio $h_T/h_{T'}$ is bounded uniformly above and below by positive constants.  Combining~(\ref{errorterm1}) and~(\ref{errorterm2}) shows that the error $g_h-g$ satisfies~(\ref{optimalorder_general}).

\printbibliography

\end{document}